\documentclass{amsart}
\usepackage[normalem]{ulem}
\usepackage{amssymb,epsfig,mathrsfs,mathpazo,scalerel,url,amsthm,thmtools,bbm}
\usepackage{xcolor}
\usepackage{stmaryrd}
\usepackage{epsfig,bm}
\usepackage[multiple]{footmisc}
\usepackage{accents} 
\usepackage{mathtools} 

\usepackage{numprint}
\usepackage{arydshln, booktabs, makecell, pbox, tabularx}

\usepackage{graphicx, caption, subcaption}
\usepackage[section]{placeins}

\newtheorem{theorem}{Theorem}[section]
\newtheorem{lemma}[theorem]{Lemma}

\declaretheorem[style=remark,qed=$\vartriangle$,sibling=theorem]{remark}
\numberwithin{equation}{section}

\newcommand{\R}{\mathbb R}
\newcommand{\C}{\mathbb C}

\newcommand{\cL}{\mathcal L}
\newcommand{\cS}{\mathcal S}
\newcommand{\cN}{\mathcal N}
\newcommand{\Lis}{\cL\mathrm{is}}

\newcommand{\identity}{\mathrm{Id}}

\DeclareMathOperator{\vol}{vol}

\DeclareMathOperator{\supp}{supp}

\DeclareMathOperator*{\argmin}{argmin}

\DeclareMathOperator{\Span}{span}

\DeclareMathOperator{\divv}{div}

\newcommand{\new}[1]{#1}

\newcommand{\be}{\begin{equation}}
\newcommand{\ee}{\end{equation}}

\newcommand{\tria}{{\mathcal T}}
\newcommand{\RT}{\mathit{RT}}

{\par\begin{samepage}%
\begin{tabbing}\ttfamily%
 \hspace*{5mm}\=\hspace{3ex}\=\hspace{3ex}\=\hspace{3ex}\=\hspace{3ex}%
\=\hspace{3ex}\=\hspace{3ex}\=\hspace{3ex}\=\hspace{3ex}\kill}%
{\end{tabbing}\end{samepage}}

\makeatletter
\@namedef{subjclassname@2020}{%
  \textup{2020} Mathematics Subject Classification}
\makeatother

\title[Preconditioning of a pollution-free discretization of the Helmholtz equation]{Preconditioning of a pollution-free discretization of the Helmholtz equation}

\date{\today}

\author{Harald Monsuur}

\address{
Korteweg--de Vries (KdV) Institute for Mathematics, University of Amsterdam, P.O. Box 94248, 1090 GE Amsterdam, The Netherlands.
}
\email{harald.monsuur@hotmail.com}

\thanks{This research has been supported by the Netherlands Organization for Scientific Research (NWO)
under contract.~no.~613.009.138. 
We acknowledge the support of SURF (www.surf.nl) in using the National Supercomputer Snellius.}


\begin{document}

\begin{abstract} 
We present a pollution-free first order system least squares (FOSLS) formulation for the Helmholtz equation, solved iteratively using a block preconditioner. This preconditioner consists of two components: one for the Schur complement, which corresponds to a preconditioner on $L_2(\Omega)$, and another defined on the test space, which we ensure remains Hermitian positive definite using subspace correction techniques. The proposed method is easy to implement and is directly applicable to general domains, including scattering problems. Numerical experiments demonstrate a linear dependence of the number of MINRES iterations on the wave number $\kappa$. We also introduce an approach to estimate algebraic errors which prevents unnecessary iterations.
\end{abstract}

\maketitle
\section{Introduction}
\subsection{The Helmholtz equation}
In this work we consider the Helmholtz equation on a bounded Lipschitz domain $\Omega\subset \R^d$. The Helmholtz equation with (mixed) Dirichlet, Neumann and/or Robin boundary conditions consists of finding $\phi \colon \Omega \to \mathbbm{C}$ that satisfies
\be \label{chap4:eq:helmholtz}
\begin{aligned}
-\Delta \phi-\kappa^2 \phi &=  f & &\text{on } \Omega,\\
\phi & =  g_D \quad & &\text{on } \Gamma_D,\\
\tfrac{\partial \phi}{\partial \vec{n}} & = g_N \quad & &\text{on } \Gamma_N,\\
\tfrac{\partial \phi}{\partial \vec{n}}- i \kappa \phi&=  g_R \quad  & &\text{on } \Gamma_R,
\end{aligned}
\ee
where $f$ is a given source term, and $g_D$, $g_N$, and $g_R$ are prescribed boundary data on the Dirichlet, Neumann, and Robin parts of the boundary, respectively.

We assume that the wave number $\kappa$ is real and positive, and that the boundary $\partial\Omega$ consists of three disjoint components $\Gamma_D, \Gamma_N$ and $\Gamma_R$ with $|\Gamma_R|>0$.
\subsection{Numerical approximation of the Helmholtz equation}
Approximating the solution to the Helmholtz problem is a difficult task, mainly for three reasons. First of all, with piecewise polynomial approximations, one needs many unknowns to approximate solutions well because of their generally oscillatory nature. Secondly, there exists the problem of pollution, that is, quasi-optimality of solutions is not always guaranteed. For the Galerkin method in particular, quasi-optimal solutions are usually only obtained under extra conditions on the mesh-size and the polynomial degree. In the seminal work \cite{202.8} it was shown that quasi-optimal solutions are obtained under the condition that $\tfrac{\kappa h}{p}$ is sufficiently small, and the polynomial degree $p$ is at least $\mathcal{O}(\log \kappa)$. Lastly, obtaining iterative solutions of the resulting matrix-vector equation is challenging since the Helmholtz problem becomes increasingly ill-conditioned for large wave numbers. Consequently, many techniques designed for elliptic problems become less effective at higher wave numbers (see \cite{MR3050918}).

Numerous methods have been proposed to solve the Helmholtz equation. One class of methods uses approximation properties of problem adapted basis functions (e.g. \cite{243.53, 138.296, MR4490294}). Another substantial class of methods is based on the approximation by piecewise polynomials, as we shall consider. An example is provided by the Discontinuous Galerkin methods generated by the Ultra-Weak Variational Formulation (UWVF) (\cite{35.9373,35.8656}).
Some methods have the benefit of producing positive definite algebraic systems, including (First Order) Least Squares methods (\cite{182.35, 38.72, 20.186}) and Discontinuous Petrov–Galerkin (DPG) methods (\cite{64.155, 75.63}). The method we consider here will lead to a saddle-point discretization, which has many of the same benefits as positive definite discrete systems, and is closely related to the DPG method.

\subsection{Ultra-weak first order system formulation, and optimal test norm.}
In this article we build upon our previous article \cite{204.18}. There the problem of pollution is tackled by introducing a first order system formulation and employing the optimal test norm. To summarize, we wrote the Helmholtz equation as an ultra-weak first order variational system $\langle \mathbbm{u}, B_\kappa' \mathbbm{v}\rangle_U=q(\mathbbm{v})$ ($ \mathbbm{v} \in V$), where $U:=L_2(\Omega)\times L_2(\Omega)^d$, and $V$ is a closed subspace of $H^1(\Omega)\times H(\divv;\Omega)$ defined by the incorporation of adjoint homogeneous boundary conditions. Here, $\mathbbm{u}=(\phi,\frac{1}{\kappa} \nabla \phi)$, with $\phi$ denoting the Helmholtz solution, $B_\kappa'$ is a partial differential operator of first order, and $q \in V'$ is a functional defined in terms of the data of the Helmholtz problem.
We demonstrated that for any $\kappa>0$ this formulation is well-posed in the sense that $B_\kappa$, i.e., the adjoint of $B_\kappa'$, is a boundedly invertible operator
 from $U$ to the dual space $V'$.
 
Still, when both $U$ and $V$ are equipped with their canonical norms, the condition number of $B_\kappa$ increases with increasing $\kappa$. To deal with this problem, we replaced the canonical norm on $V$ by the so-called \emph{optimal test norm} $\|\cdot\|_{V_\kappa}:=\|B'_\kappa\cdot\|_U$, and equipped $V'$ with the resulting dual norm. This modification makes $B_\kappa$ an isometry. 
Consequently, given a finite-dimensional subspace $U^\delta \subset U$, the least squares approximation $\hat{\mathbbm{u}}^\delta:=\argmin_{\mathbbm{w}^\delta \in U^\delta}\|q-B_\kappa \mathbbm{w}^\delta\|_{V_\kappa'}$ is the \emph{best} approximation to $\mathbbm{u}$ from $U^\delta$ with respect to the norm on $U$. Here $\delta$, which refers to 'discrete', serves as an index that indicates that $U^\delta$ belongs to a family of finite-dimensional subspaces of $U$.
The use of the optimal test norm has been advocated in \cite{64.155,45.44,35.8565}, and can already be found in \cite{19.83}.

\subsection{`Practical' method}
Since $V_\kappa$ (we write $V_\kappa$ instead of $V$ to emphasize the use of the $\kappa$-dependent norm) is an infinite-dimensional space, we cannot compute the residual minimizer w.r.t.~the norm $\sup_{0 \neq \mathbbm{v}\in V_\kappa} \frac{|\cdot(\mathbbm{v})|}{\|B_\kappa' \mathbbm{v}\|_{U}}$.
By replacing the supremum over $\mathbbm{v}\in V_\kappa$ by a supremum over $\mathbbm{v}^\delta\in V_\kappa^\delta$ for some (sufficiently large) finite-dimensional subspace $V_\kappa^\delta \subset V_\kappa$
we obtain an implementable `practical' method. Its solution is obtained as the second component of the pair $(\mathbbm{v}^\delta,\mathbbm{u}^\delta) \in V_\kappa^\delta \times U^\delta$ that solves 
\be \label{chap4:zadel}
\left\{
\begin{array}{lcll}
\langle B_\kappa'\mathbbm{v}^\delta, B_\kappa'\undertilde{\mathbbm{v}}^\delta\rangle_U +
\langle \mathbbm{u}^{\delta}, B_\kappa'\undertilde{\mathbbm{v}}^\delta\rangle_U 
& \!\!= \!\!& q(\undertilde{\mathbbm{v}}^\delta) & (\undertilde{\mathbbm{v}}^\delta \in V_\kappa^\delta),\\
\langle B_\kappa'\mathbbm{v}^\delta,\undertilde{\mathbbm{u}}^\delta\rangle_U
& \!\!=\!\! & 0 & (\undertilde{\mathbbm{u}}^\delta \in U^\delta).
\end{array}
\right.
\ee
We refer to $V^\delta_\kappa$ and $U^\delta$ as the test and trial space. Since the above system is a First Order System Least Squares method in disguise we will refer to the above method as the FOSLS method.

If the inf-sup constant
$\gamma^\delta_\kappa=\gamma^\delta_\kappa(U^\delta,V_\kappa^\delta):= \inf_{0 \neq \undertilde{\mathbbm{u}}^\delta \in U^\delta} \sup_{0 \neq \undertilde{\mathbbm{v}}^\delta \in V_\kappa^\delta} \frac{|
\langle \undertilde{\mathbbm{u}}^{\delta}, B_\kappa'\undertilde{\mathbbm{v}}^\delta\rangle_U|}{\|\undertilde{\mathbbm{u}}^\delta\|_U\|B_\kappa' \undertilde{\mathbbm{v}}^\delta\|_{U}}$, is positive, the above system has a unique solution and it holds that 
$$
\|\mathbbm{u}-\mathbbm{u}^\delta\|_U \leq \tfrac{1}{\gamma_\kappa^\delta} \inf_{\undertilde{\mathbbm{u}}^\delta \in U^\delta} \|\mathbbm{u}-\undertilde{\mathbbm{u}}^\delta\|_U.
$$
If, for some constant $C>0$, one chooses $V_\kappa^\delta=V_\kappa^\delta(U^\delta)$ such that $\gamma_\kappa^\delta\geq C$ for any $\delta$, the solutions $u^\delta$ are quasi-best, i.e. the method is pollution-free.

In \cite{204.18}, for the ideal case of convex polygonal domains with Robin boundary conditions, we investigated how this can be done for trial spaces of the form $U^\delta = \cS_p^{-1}(\tria^\delta)^{d+1}$ (i.e. spaces of discontinuous piecewise polynomials of degree $p$). We showed that it suffices for the polynomial degree $\tilde p$ of the test space $V^\delta_\kappa$ to be proportional to $\max (\log \kappa, p^2)$, where the mesh for $V^\delta_\kappa$ coincides with that of $U^\delta$, apart from a slight refinement near the corners of the domain. 
As a complementary topic in this article, we will investigate the dependence of $\gamma_\kappa^\delta$ on $p,\tilde p$ and $\kappa$ for domains with general boundary conditions.

\subsection{Solving the algebraic system}
The main topic of this paper is the iterative solution of the algebraic system ~\eqref{chap4:zadel}.
There is a wide range of research done on obtaining the solution to algebraic systems arising from the Helmholtz equation using iterative methods. Due to the ill-conditioned nature of the Helmholtz equation, many standard iterative methods are ineffective ~\cite{MR3050918}. To resolve this problem, many preconditioning methods have been designed, including domain decomposition methods ~\cite{MR4286258, MR3647952, MR2277034, MR3647407, MR3359680}, shifted Laplacian methods ~\cite{MR3395145, MR3084825, MR2387515}, and sweeping domain preconditioners ~\cite{MR2789492, MR2818416, MR3085122, MR4126398, MR4196155, MR2199758, MR3179763}. There also exist two-grid methods ~\cite{MR3285899, MR3084826} and multigrid methods which include wave-ray corrections ~\cite{MR1615163, 182.35}. For the DPG method, there also exist multilevel preconditioners ~\cite{MR4630864, MR3630816}.

In this work, we opt to use a block preconditioner for our saddle-point formulation. The Schur complement of our formulation corresponds to a uniformly boundedly invertible operator on $L_2(\Omega)^{d+1}$, for which a preconditioner can be easily devised. The upper left block corresponds to the inner product on $V_{\kappa}$ and is therefore positive definite. In this work we attempt to build an efficient preconditioner for this upper left block based on successive subspace corrections~\cite{MR1193013}. 

Since we devise a block preconditioner using Hermitian positive definite operators for a saddle-point formulation, we can make use of the preconditioned MINRES method \cite{MR383715}. 

For iterative solvers designed for saddle-point systems, convergence can be guaranteed in terms of the eigenvalues of the preconditioned system, provided that the block-preconditioner is Hermitian positive definite. This property offers a significant advantage over solution techniques based on Galerkin discretization, which produce indefinite systems, that are even non-Hermitian when Robin boundary conditions are present. In such cases, one often resorts to GMRES, where the choice of a suitable preconditioner becomes more delicate and may depend on problem-specific parameters. In contrast, our preconditioner can be directly applied to any problem, including so-called \emph{scattering problems}. 

\subsection{Adaptivity and error estimation}
One major advantage of the FOSLS method is its applicability to (adaptively refined) non-quasi-uniform meshes. Using the computed residual $\|B'_\kappa \mathbbm{v}^\delta\|_U$ from the 'practical' method as an estimator for the total error $\|\mathbbm{u}-\mathbbm{u}^\delta\|_U$, and by splitting this residual into local contributions, we can apply Dörfler marking to drive adaptive mesh-refinement. We only need solutions with sufficient accuracy on intermediate meshes, which makes the process of creating a suitable mesh quite efficient.

\subsection{Organisation}
In Section \ref{chap4:sec:helm} we summarize findings about the ultra-weak first order system formulation of the Helmholtz equation from \cite{204.18}. We first write the Helmholtz equation as a first order system. By applying integration-by-parts, we derive an ultra-weak formulation which corresponds to a map from $U=L_2(\Omega) \times L_2(\Omega)^d$ to the dual of a space $V_\kappa$. Thanks to $U$ being a product of $L_2(\Omega)$-spaces, we can use the \emph{optimal test norm} to obtain \emph{pollution-free} approximations. This \emph{ideal} method, which cannot be implemented, is then replaced by a \emph{practical} method. Finally, we discuss error estimation and adaptivity. In Section \ref{chap4:sec:Iterative solvers} first we discuss iterative methods and block-preconditioners. We then devise two preconditioners, one for the Schur complement of our system and one for the upper left block. In Section \ref{chap4:sec:numerics} we present some numerical results. 

\subsection{Notations}
For normed linear spaces $E$ and $F$, by $\cL(E, F)$ we denote the normed linear
space of bounded linear mappings $E \rightarrow F$, and by $\Lis(E, F)$ its subset of boundedly invertible linear mappings $E \rightarrow F$. We write $E \hookrightarrow F$ to denote that $E$ is
continuously embedded into $F$. Since we consider linear spaces over $\C$, for a normed linear space $E$ its dual $E'$ is the normed linear space of anti-linear functionals.

By $C \lesssim D$ we will mean that $C$ can be bounded by a multiple of $D$, unless explicitly stated otherwise, \emph{independently} of parameters which $C$ and $D$ may depend on, such as the wave number $\kappa$ or the discretisation index $\delta$.
Furthermore, $C \gtrsim D$ is defined as $D \lesssim C$, and $C\eqsim D$ as $C\lesssim D$ and $C \gtrsim D$.

Let $\Omega\subset\mathbb{R}^n$ be an open bounded set. Let $\mathcal{T}^\delta$ be a conforming partition of $\Omega$ into $n$-simplices. For such a partition we will write $\mathcal{N}^\delta$ for the set of vertices. We write $S_p^{-1}(\mathcal{T}^\delta):=\{u\in L_2(\Omega)\colon u|_K\in P_p(K)\}$ and $S_p^0(\mathcal{T}^\delta) = S_p^{-1}(\mathcal{T}^\delta)\cap H^1(\Omega)$. The well-known Raviart-Thomas finite element space of order $p$ on $\mathcal{T}^\delta$ is denoted by $\RT_p(\mathcal{T}^\delta)$. 

\section{Helmholtz equation} \label{chap4:sec:helm}
We recall the Helmholtz equation given in \eqref{chap4:eq:helmholtz}. For $\Omega\subset \R^d$ being a bounded Lipschitz domain, the Helmholtz equation with (mixed) Dirichlet, Neumann and/or Robin boundary conditions consists of finding $\phi\in H^1(\Omega)$ that satisfies
\be
\begin{aligned}\label{chap4:eq:helmholtz-scaled}
-\Delta \phi-\kappa^2 \phi &=  \kappa^2 f & &\text{on } \Omega,\\
\phi & =  \kappa g_D \quad & &\text{on } \Gamma_D,\\
\tfrac{\partial \phi}{\partial \vec{n}} & = \kappa^2 g \quad & &\text{on } \Gamma_N,\\
\tfrac{\partial \phi}{\partial \vec{n}}- i \kappa \phi&=  \kappa^2 g \quad  & &\text{on } \Gamma_R,
\end{aligned} 
\ee
where we added some harmless scalings on the righthand sides that are made for convenience.
We assume that the wave number $\kappa$ is real and positive, and the boundary $\partial\Omega$ consists of three disjoint components $\Gamma_D, \Gamma_N$ and $\Gamma_R$ with $|\Gamma_R|>0$.

Here we assume that $f \in H^1_{0,\Gamma_D}(\Omega)'$, $g_D \in H^{\frac12}(\Gamma_D)$,
and $g \in H^{-\frac12}(\Gamma_N \cup \Gamma_R)$ (=$H_{00}^{\frac12}(\Gamma_N \cup \Gamma_R)'$). 
\subsection{Ultra-weak first order formulation}
To arrive at a first order formulation we first decompose $f \in  H^1_{0,\Gamma_D}(\Omega)'$ using the Riesz' representation theorem. For some $f_1 \in L_2(\Omega)$ and $\vec{f}_2 \in L_2(\Omega)^d$, we write 
$$
f(\eta)=\int_\Omega f_1 \overline{\eta} + \tfrac{1}{\kappa} \vec{f}_2 \cdot \nabla \overline{\eta}\,dx \quad (\eta \in H^1_{0,\Gamma_D}(\Omega)).
$$
Then introducing $\vec{u}=\tfrac{1}{\kappa} \nabla \phi - \vec{f}_2$ we rewrite \eqref{chap4:eq:helmholtz-scaled} as a \emph{first order system}
\be \label{chap4:1storder}
\begin{aligned}
-\tfrac{1}{\kappa} \divv \vec{u}- \phi&= f_1 & &\text{on } \Omega,\\
\tfrac{1}{\kappa} \nabla \phi- \vec{u}&=\vec{f}_2 & &\text{on } \Omega,\\
\phi & = \kappa g_D \quad  & &\text{on } \Gamma_D,\\
\vec{u} \cdot \vec{n} & = \kappa g & &\text{on } \Gamma_N,\\
\vec{u} \cdot \vec{n} - i \phi&=\kappa g \quad & &\text{on } \Gamma_R.
\end{aligned}
\ee
The ultra-weak first order formulation is subsequently obtained by moving all the derivatives to some test functions. These test functions are $\eta$ and $\vec{v}$ with 
$\eta=0$ on $\Gamma_D$, $\vec{v}\cdot \vec{n}=0$ on $\Gamma_N$, and
$\vec{v}\cdot\vec{n} + i \eta =0$ on $\Gamma_R$, with which we test the first and second equation respectively. After integration-by-parts and substituting the boundary conditions we arrive at the \emph{ultra-weak variational formulation}, in which all the boundary conditions are \emph{natural}:
\be \label{chap4:eq:ultra-weak}
\begin{split}
\big(B_\kappa&(\phi,\vec{u})\big)(\eta,\vec{v}):=
\int_\Omega  \tfrac{1}{\kappa} \vec{u}\cdot\nabla \overline{\eta} -  \phi \overline{\eta}- \tfrac{1}{\kappa} \phi \divv \overline{\vec{v}}- \vec{u}\cdot \overline{\vec{v}}\,dx\\
&=\int_\Omega f_1 \overline{\eta}+\vec{f}_2 \cdot \overline{\vec{v}}\,dx
-\int_{\Gamma_D} g_D \overline{\vec{v}}\cdot\vec{n}  \,ds+ \int_{\Gamma_N \cup \Gamma_R} g \overline{\eta} \,ds=:q(\eta,\vec{v}).
\end{split}
\ee
It was shown in \cite{204.18} that $B_\kappa$ is a boundedly invertible mapping from $U$ to $V$ defined below. 
\begin{theorem}
For 
\begin{align*}
&U:=L_2(\Omega) \times L_2(\Omega)^d,
\intertext{and}
&V:=\Big\{(\eta,\vec{v}) \in H_{0,\Gamma_D}^1(\Omega) \times H(\divv;\Omega)\colon \\
&\hspace*{9em}\int_{\partial\Omega} \vec{v}\cdot\vec{n} \,\overline{\psi} \,ds + i \int_{\Gamma_R} \eta \overline{\psi} \,ds=0 \quad (\psi\in H_{0,\Gamma_D}^1(\Omega))\Big\},
\end{align*}%
both being Hilbert spaces equipped with their canonical norms $\|\cdot\|_U:=\|\cdot\|_{L_2(\Omega) \times L_2(\Omega)^d}$ and $\|\cdot\|_V:=\|\cdot\|_{H^1(\Omega) \times H(\divv;\Omega)}$, it holds that $B_\kappa \in \Lis(U,V')$.
\end{theorem}
\begin{remark}\label{chap4:remark:robin-discrete}
    For $(\eta^\delta, \vec{v}^\delta)\in V^\delta = \cS_{\tilde{p}}^0(\tria^\delta)\times\RT_{\tilde{p}}(\tria^\delta)$ to be a member of $V$, it simply means that $\vec{v}^\delta\cdot \vec{n}=0$ on $\Gamma_N$, $\eta^\delta = 0$ on $\Gamma_D$ and $\vec{v}^\delta\cdot\vec{n} + i\eta^\delta = 0$ on $\Gamma_R$. By applying standard bases for $\cS^0_{\tilde{p}}(\tria^\delta)$ and $\RT_{\tilde{p}}(\tria^\delta)$, a basis for $V^\delta$ is obtained when one removes the usual DoFs of $\RT_{\tilde{p}}(\tria^\delta)$ associated to element faces on $\Gamma_N$ and the DoFs of $S_{\tilde{p}}^0(\mathcal{T}^\delta)$ associated to the nodes on $\Gamma_D$, and one eliminates the DoFs of $\RT_{\tilde{p}}(\tria^\delta)$ associated to element faces on $\Gamma_N \cup \Gamma_R$ by imposing $\vec{v}\cdot\vec{n}=- i \eta$ with $\eta\in S_{\tilde{p}}^0(\mathcal{T}^\delta)$.

\end{remark}

\subsection{Pollution-free least squares approximations}\label{chap4:section: optimal test norm}
The least squares method approximates the solution of the system $B_\kappa(\mathbbm{u}) = B_\kappa(\phi,\vec{u})=q$ by minimizing the residual over some finite-dimensional subspace $U^\delta \subset U$, i.e. one defines
$$
\bar{\mathbbm{u}}^{\delta} := \argmin_{\mathbbm{w}^\delta \in U^\delta} \|q-B_\kappa \mathbbm{w}^\delta\|_{V'}.
$$ 

However, this approach is not attractive on its own. We only have available the estimate $$\|\mathbbm{u} - \bar{\mathbbm{u}}^\delta\|_U\leq \|B_\kappa\|_{U\to V'}\|B_\kappa^{-1}\|_{V'\to U} \min_{\mathbbm{w}^\delta\in U^\delta}\|\mathbbm{u} - \mathbbm{w}^\delta\|_U,$$
whereas the condition number $\|B_\kappa\|_{U\to V'}\|B_\kappa^{-1}\|_{V'\to U}$ cannot be expected to be small for large $\kappa$.

To circumvent the issue of large condition numbers, we replace the norm on the space $V$ with the $\kappa$-dependent \emph{optimal test norm} $\|\cdot\|_{V_\kappa}$, and the norm of the dual space $V'$ with the induced dual norm $$\|\cdot\|_{V_\kappa'}:=\sup_{0\not =v\in V}\tfrac{\cdot(v)}{\|v\|_{V_\kappa}}.$$

This \emph{optimal test norm} is given by
\be \label{chap4:eq:optimal test norm}
\|\mathbbm{v}\|_{V_{\kappa}}:= \|B'_\kappa \mathbbm{v}\|_{U'} = \|B'_\kappa \mathbbm{v}\|_{U},
\ee
where we use that $U=L_2(\Omega)^{d+1}\simeq U'$,
and the corresponding inner product is given by $\langle\cdot,\cdot\rangle_{V_{\kappa}} = \langle B_\kappa '\cdot, B_\kappa '\cdot \rangle_U$. 
The norm $\|\cdot\|_{V_\kappa}$, is chosen in such a way that the operator $B_\kappa$ becomes an \emph{isometry}, i.e. $\|B_\kappa \mathbbm{u}\|_{V_{\kappa}'}=\|\mathbbm{u}\|_U$. In other words, with the optimal test norm, the operator $B_\kappa$ has a condition number equal to one.

Consequently, for any $q \in V'$, and any closed, subspace $\{0\} \subsetneq U^\delta \subset U$, the least squares solution
\be \label{chap4:ls}
\hat{\mathbbm{u}}^{\delta}:=\argmin_{\mathbbm{w}^\delta \in U^\delta} \|q- B_\kappa \mathbbm{w}^\delta\|_{V_{\kappa}'}
\ee
is the \emph{best} approximation to $\mathbbm{u}$ from $U^\delta$ w.r.t. $\|\cdot\|_U$. 

To emphasize the use of the $\kappa$-dependent optimal test norm, from now on we will write $V_\kappa$ for the space $V$ equipped with the $\kappa$-dependent optimal test norm and write $V_\kappa'$ for its dual, which is equipped with the norm $\|\cdot \|_{V_\kappa'}$.
With $R_\kappa\colon V_\kappa'\to V_\kappa$ being the Riesz lifting operator defined by $\mathbbm{f}(\mathbbm{v}) = \langle R_\kappa \mathbbm{f}, \mathbbm{v}\rangle_{V_\kappa}$, $(\mathbbm{f}\in V', \mathbbm{v}\in V)$, the corresponding dual inner product is given by $\langle\,\mathbbm{f},\mathbbm{q}\,\rangle_{V_{\kappa}'}:=\langle R_\kappa\mathbbm{f}, R_\kappa\mathbbm{q}\rangle_{V_\kappa}$. 
\subsubsection{'Practical' method}
The solution $\hat{\mathbbm{u}}^{\delta}$ to \eqref{chap4:ls} is not computable because we cannot evaluate the dual norm ${\|\cdot \|_{V_\kappa'}}$. To deal with this problem, we replace $V_\kappa$ with a finite-dimensional subspace $V_\kappa^\delta \subset V_\kappa$ that satisfies
\begin{align}\label{chap4:inf-sup definition}
\gamma^\delta_\kappa:=\inf_{0 \neq \mathbbm{u}^\delta \in U^\delta} \sup_{0 \neq \mathbbm{v}^\delta \in V_\kappa^\delta} \frac{|
(B_\kappa \mathbbm{u}^\delta)(\mathbbm{v}^\delta)|}{\|\mathbbm{u}^\delta\|_U\|\mathbbm{v}^\delta\|_{V_{\kappa}}} > 0,
\end{align}
and instead solve the \emph{practical} least squares problem
\be \label{chap4:ls-discrete}
\mathbbm{u}^{\delta}:=\argmin_{\mathbbm{w}^\delta \in U^\delta} \sup_{\undertilde{\mathbbm{v}}^\delta\in V_\kappa^\delta} \frac{|(q- B_\kappa \mathbbm{w}^\delta)(\undertilde{\mathbbm{v}}^\delta)|}{\|\undertilde{\mathbbm{v}}^\delta\|_{V_{\kappa}}} = \argmin_{\mathbbm{w}^\delta \in U^\delta}\|q-B_\kappa \mathbbm{w}^\delta\|_{{V_{\kappa}^\delta}'}.
\ee
The effect of this discretization of the test space is characterized by the theorem below (see \cite{35.8565,204.18}).
A consequence of this theorem is that if we choose the test space $V_{\kappa}^\delta\subset V$ large enough w.r.t. the trial space $U^\delta\subset U$, the solution $\mathbbm{u}^\delta$ is a \emph{quasi}-best approximation to $\mathbbm{u}$ from $U^\delta$ also known as a \emph{pollution-free} approximation.
\begin{theorem} \label{chap4:theorem:quasi-optimality}
If $\gamma^\delta_\kappa>0$, then for every $q\in V_\kappa'$, the system \eqref{chap4:ls-discrete} has a unique solution and 
$$
\sup_{\mathbbm{u} \in U\setminus U^\delta}
\frac{
\|\mathbbm{u}-\mathbbm{u}^\delta\|_U
}{
\inf_{\undertilde{\mathbbm{u}}^\delta \in U^\delta} \|\mathbbm{u}-\undertilde{\mathbbm{u}}^\delta\|_U
}= \frac{1}{\gamma^\delta_\kappa},
$$
i.e. $ \frac{1}{\gamma^\delta_\kappa}$ is the worst possible `pollution factor'.
\end{theorem}

Of course, the question remains how to choose the test space $V_\kappa^\delta$ such that the pollution factor is bounded uniformly in $\delta$ and $\kappa$. In \cite{204.18}, for convex domains $\Omega$ with $\Gamma_R = \partial \Omega$ and quasi-uniform meshes, choices of $U^\delta$ and $V_\kappa^\delta$ were found under which the pollution factor $\gamma_\kappa^\delta$ is bounded away from zero uniformly in the mesh-size and wave number $\kappa$. If $U^\delta = \cS_p^{-1}(\tria^\delta)^{d+1}$, it suffices to choose the test space $V_{\kappa}^\delta:=(\cS_{\tilde{p}}^0(\tilde{\tria}^\delta) \times \RT_{\tilde{p}}(\tilde{\tria}^\delta)) \cap V$, where $\tilde{\tria}^\delta$ is a slightly refined mesh w.r.t. $\tria^\delta$, and $\tfrac{\log \kappa}{\tilde{p}}$ and $\tfrac{p^2}{\tilde{p}}$ are sufficiently small. 

The results obtained in \cite{204.18} rely on results from \cite{202.8} concerning the approximability of the solutions of adjoint Helmholtz problems by finite element functions. A key ingredient in their analysis is the assumption that a norm of the solution operator for the Helmholtz problem satisfies a polynomial-in-$\kappa$ bound.
Such bounds have been established for some classes of Helmholtz problems in \cite{MR2352336,MR2194984}. These results, however, are not available when $\Gamma_R\not =\partial \Omega$.

In principle, to obtain sufficiently large inf-sup constants, either one could define the test space on a mesh that is refined with respect to $\tria^\delta$ or one could increase the polynomial order $\tilde{p}$ of the test space. The first option is not as attractive from an implementation point of view, which is why we opt for the second option. In Section \ref{chap4:sec:numerics} we perform a numerical investigation concerning a satisfactory choice of $\tilde{p}$. For the more challenging cases where $\Omega$ is not convex and $\Gamma_R \not = \partial \Omega$, we numerically observe that quasi-optimality of the numerical approximation can still be achieved by choosing $\tilde{p}$ large enough. However, the dependence of $\tilde{p}$ on the wave number and $p$ can be worse. 

In the remainder of this article we set 
\be \label{chap4:FEspace}
\begin{aligned}
U^\delta = &\mathcal{S}_p^{0}(\mathcal{T}^\delta)^{d+1},\\
   V_{\kappa}^\delta=(\cS_{\tilde{p}}^0(&\tria^\delta) \times \RT_{\tilde{p}}(\tria^\delta)) \cap V_\kappa.
\end{aligned}
\ee
\begin{remark}
Instead of choosing a continuous trial space one could choose $U^\delta=\mathcal{S}_p^{-1}(\mathcal{T}^\delta)^{d+1}$, i.e. $U^\delta$ is \emph{discontinuous} across edges/faces. The reason we opt otherwise is two-fold. Firstly, the approximation quality is the same for both continuous and discontinuous spaces, but the inf-sup constant $\gamma_\kappa^\delta$ is always larger for continuous trial spaces. Secondly, the boosted method and the error estimator, which will both be introduced in Section \ref{chap4:sec:errorestimation}, perform better in case of continuous trial spaces, see Remark \ref{chap4:continuousvsdiscontinuous}.
\end{remark}

\subsection{Euler-Lagrange equations}
Finally, the formulation we use in computations arises when considering the Euler-Lagrange equations of \eqref{chap4:ls-discrete}. These equations read as finding $\mathbbm{u}^\delta\in U^\delta$ that satisfies
\begin{align}\label{chap4:eq:euler-lagrange}
    \langle q-B_\kappa \mathbbm{u}^\delta, B_\kappa\undertilde{\mathbbm{w}}^\delta\rangle_{{V_\kappa^\delta}'} = 0, \mbox{ for any } \undertilde{\mathbbm{w}}^\delta\in U^\delta. 
\end{align} 
Because we cannot easily evaluate the dual inner product $\langle \cdot,\cdot \rangle_{{V_{\kappa}^\delta}'}$, we introduce a new variable $\mathbbm{v}^\delta\in  V_\kappa^\delta$ which is the Riesz-lift of $\mathbbm{q}-B_\kappa \mathbbm{u}^\delta\in (V_\kappa^\delta)'$, i.e.
\begin{align}\label{chap4:A}
   \langle \mathbbm{v}^\delta, \undertilde{\mathbbm{v}}^\delta\rangle_{V_\kappa} = \langle B_\kappa'\mathbbm{v}^\delta, B_\kappa'  \undertilde{\mathbbm{v}}^\delta\rangle_U = (q-B_\kappa \mathbbm{u}^\delta)(\undertilde{\mathbbm{v}}^\delta)\mbox{ for }\undertilde{\mathbbm{v}}^\delta \in V_{\kappa}^\delta. 
\end{align}
Now letting $R_\delta\in \Lis({V_\kappa^\delta}', V_\kappa^\delta)$ be the Riesz lifting operator, defined by $\langle R_\delta \mathbbm{f}, \undertilde{\mathbbm{v}}^\delta\rangle_{V_\kappa} = \mathbbm{f}(\undertilde{\mathbbm{v}}^\delta)$ for $\undertilde{\mathbbm{v}}^\delta \in V_\kappa^\delta$, we find that 
\begin{align}\begin{split}\label{chap4:B}
\langle B_\kappa'\mathbbm{v}^\delta,\undertilde{\mathbbm{u}}^\delta\rangle_U &= (B_\kappa \undertilde{\mathbbm{u}}^\delta)(\mathbbm{v}^\delta) \\
	&= \langle R_\delta B_\kappa \undertilde{\mathbbm{u}}^\delta, \mathbbm{v}^\delta\rangle_{V_\kappa}\\
	&=\overline{ (q-B_\kappa\mathbbm{u}^\delta)(R_\delta B_\kappa \undertilde{\mathbbm{u}}^\delta)}\\
	&= \overline{\langle  (q-B_\kappa \mathbbm{u}^\delta), B_\kappa \undertilde{\mathbbm{u}}^\delta\rangle_{{V_\kappa^\delta}'}}\\
	&= 0,
    \end{split}
\end{align}
thanks to $\eqref{chap4:eq:euler-lagrange}$. 

Putting \eqref{chap4:A} and \eqref{chap4:B} together, we conclude that the pair $(\mathbbm{v}^\delta,\mathbbm{u}^{\delta})\in V_\kappa^\delta\times U^\delta$ solves the saddle-point system
\be \label{chap4:saddle-discrete}
\left\{\hspace*{-0.5em}
\begin{array}{lcll}
\langle B_\kappa'\mathbbm{v}^\delta, B_\kappa'\undertilde{\mathbbm{v}}^\delta\rangle_U +
\langle \mathbbm{u}^{\delta}, B_\kappa'\undertilde{\mathbbm{v}}^\delta\rangle_U 
& \!\!= \!\!& q(\undertilde{\mathbbm{v}}^\delta) & (\undertilde{\mathbbm{v}}^\delta \in V_\kappa^\delta),\\
\langle B_\kappa'\mathbbm{v}^\delta,\undertilde{\mathbbm{u}}^\delta\rangle_U
& \!\!=\!\! & 0 & (\undertilde{\mathbbm{u}}^\delta \in U^\delta).
\end{array}
\right.\hspace*{-1.2em}
\ee

Section \ref{chap4:sec:Iterative solvers} is devoted to solving this saddle-point system.
\subsection{A posterior error estimation and boosted approximation.} \label{chap4:sec:errorestimation}
Even though the function $\mathbbm{v}^\delta$ is not of main interest, it can be used to improve the solution $\mathbbm{u}^\delta$ and to estimate the error of the approximation. The theorem below states that the 'boosted' FOSLS approximation $(\phi_{\rm bst}^\delta,\vec{u}_{\rm bst}^\delta):=\mathbbm{u}^\delta + B_\kappa'\mathbbm{v}^\delta$ has at least the same quality as $\mathbbm{u}^\delta$, and that the \emph{error estimator} $\|B_\kappa'\mathbbm{v}^\delta\|_U$ provides a lower bound for the error $\|\mathbbm{u}-\mathbbm{u}^\delta\|_U$.
\begin{theorem}[\cite{204.18}]\label{chap4:thm:boosted} It holds that
\be \label{chap4:201}
\|\mathbbm{u}-\mathbbm{u}^\delta\|_U^2 = \|\mathbbm{u}-(\mathbbm{u}^\delta+B_\kappa'\mathbbm{v}^\delta)\|_U^2+\|B_\kappa'\mathbbm{v}^\delta\|_U^2,
\ee
which implies $$\|B_\kappa'\mathbbm{v}^\delta\|_U^2\leq\|\mathbbm{u}-\mathbbm{u}^\delta\|_U^2.$$
Furthermore, we have the estimate
$$
\|\mathbbm{u}-(\mathbbm{u}^\delta+B_\kappa'\mathbbm{v}^\delta)\|_U \leq \frac{1}{\gamma_\kappa^\delta} \,\,\inf_{\mathbbm{w}^\delta \in U^\delta+ (B_\kappa' V^\delta_{\kappa} \cap (U^\delta)^{\perp})} \|\mathbbm{u}-\mathbbm{w}^\delta\|_U.
$$
\end{theorem}
The numerical experiments from \cite{204.18} suggest that the upper bound $\|u-\mathbbm{u}^\delta\|_U \lesssim \|B_\kappa'\mathbbm{v}^\delta\|_U$ also holds when the number of degrees of freedom in $V_\kappa^\delta$ per wavelength start to exceed 1, which makes the error estimator reliable \emph{and} efficient. Furthermore, since the error estimator can be easily localized into element-wise error indicators it can be used to drive an adaptive scheme with Dörfler marking.
\begin{remark}\label{chap4:continuousvsdiscontinuous}
    The space $U^\delta+ (B_\kappa' V^\delta_{\kappa} \cap (U^\delta)^{\perp})$ is hard to analyze. For $p=1, \tilde{p}=3$, one can observe improved convergence rates for the boosted method. Surprisingly, if we instead choose $U^\delta=\mathcal{S}_p^{-1}(\mathcal{T}^\delta)^{d+1}$, i.e. $U^\delta$ is \emph{discontinuous} across edges/faces, this improved convergence rate is lost, which is probably due to $(U^\delta)^\perp$ being quite small. For this reason we prefer continuous trial spaces over discontinuous trial spaces.
\end{remark}

\section{Iterative solvers}\label{chap4:sec:Iterative solvers}
In this section we investigate iterative solvers for \eqref{chap4:saddle-discrete}.
Recall the definition of $U^\delta$ and $V_\kappa^\delta$ in \eqref{chap4:FEspace}. We equip both spaces with the finite element bases $\Phi:=\{\varphi_1,\varphi_2,\ldots\}$ and $\Psi:=\{\psi_1,\psi_2,\ldots\}$, respectively as discussed in Remark \ref{chap4:remark:robin-discrete}. We define the matrices ${\bf M}^U$, ${\bf M}^{V_{\kappa}}$, ${\bf B}_\kappa$ by
${\bf M}^U_{i j}=\langle \varphi_j,\varphi_i \rangle_U$, 
${\bf M}^{V_{\kappa}}_{i j}=\langle B_\kappa' \psi_j,B_\kappa' \psi_i \rangle_U$,
$({\bf B}_\kappa)_{i j}=\langle \varphi_j , B_\kappa' \psi_i \rangle_U$. We define the vector ${\bf q}$ by ${\bf q}_i = q(\psi_i)$. The Schur complement is defined as ${\bf S}_\kappa := {\bf B}_\kappa^H({\bf M}^{V_{\kappa}})^{-1}{\bf B}_\kappa$.

Then \eqref{chap4:saddle-discrete} can be written as
\begin{align}\label{chap4:discrete}
{\bf K} = \begin{pmatrix}
{\bf M}^{V_{\kappa}} & {\bf B}_\kappa \\
{\bf B}_\kappa^H & {\bf 0}
\end{pmatrix} \begin{pmatrix} {\bf v} \\ {\bf u} \end{pmatrix}= \begin{pmatrix} {\bf q} \\ {\bf 0} \end{pmatrix},
\end{align}
which is a saddle-point system. For systems of this form, there exist excellent iterative solution methods. See for example the use of the preconditioned MINRES method \cite{MR383715}, BPCG method \cite{MR917816}, and the inexact Uzawa method \cite{MR1031439}, discussed in \cite{MINRES_UZAWA} in the context of the Stokes equation. A huge advantage of these methods is the minimal memory requirement. For example, the preconditioned MINRES method only requires to store a few vectors thanks to the three-term recurrence relation in the Lanczos algorithm. Furthermore, convergence of the MINRES method is guaranteed without the need to properly select parameters.

 These iterative solvers for saddle-point equations are accelerated by Hermitian positive definite preconditioners ${\bf Q}_S$, ${\bf Q}_{V_\kappa}$ for the Schur complement ${\bf S}_\kappa$ and the matrix ${\bf M}^{V_\kappa}$ respectively, that satisfy
\begin{equation}\label{chap4:spectral-equivalence}
\begin{aligned}
\gamma_{S_\kappa}{\bf Q}_S &\leq {\bf S}_\kappa \leq \Gamma_{S_\kappa} {\bf Q}_S\\
\gamma_{V_\kappa}{\bf Q}_{V_{\kappa}} &\leq {\bf M}^{V_\kappa} \leq \Gamma_{V_{\kappa}}{\bf Q}_{V_{\kappa}},
\end{aligned}
\end{equation}
or equivalently meaning that the spectrum of ${\bf Q}_S^{-1}{\bf S}_\kappa$ and ${\bf Q}_{V_\kappa}^{-1}{\bf M}^{V_\kappa}$ is contained in $[\gamma_{S_\kappa},\Gamma_{S_\kappa}]$ and $[\gamma_{V_\kappa}, \Gamma_{V_\kappa}]$, respectively.
The convergence rate of the preconditioned iterative methods depends on the above positive constants. In \cite{MINRES_UZAWA, book_Elman}, this dependence is studied in more depth. 

In the remainder of this section, we provide examples of these preconditioners for both the Schur complement and the mass matrix on $V^\delta_{\kappa}$. In our case it turns out that $\Gamma_{S_\kappa}\eqsim 1$ and $\Gamma_{V_\kappa}=1$ and $\gamma_{S_\kappa}\eqsim (\gamma_\kappa^\delta)^2$. There will be no theoretical results on the lower bound $\gamma_{V_\kappa}$; we provide only some numerical insights in Section \ref{chap4:sec:numerics}.

\subsection{Schur complement}

Finding a good preconditioner for the Schur complement is rather straightforward under the assumption that the inf-sup constant $\gamma^\delta_\kappa$ is uniformly bounded from below. This assumption is already necessary to obtain pollution-free approximations, as discussed in Theorem \ref{chap4:theorem:quasi-optimality}. The next lemma establishes a connection between the Schur complement ${\bf S}_\kappa$ and the mass matrix ${\bf M}^U$, suggesting that the preconditioner ${\bf Q}_S$ should resemble ${\bf M}^U$.
\begin{lemma}\label{chap4:lemma:prec schur complement}
It holds that 
\begin{align}\label{chap4:eq:conditioning Schur}
(\gamma^\delta_\kappa)^2{\bf M}^U\leq {\bf S}_\kappa\leq {\bf M}^U.
\end{align}
Furthermore, the lower bound is sharp, i.e. there is a ${\bf z}\in \mathbb{C}^{|\Phi|}$ such that $(\gamma_\kappa^\delta)^2{\bf z}^H{\bf M}^U{\bf z} = {\bf z}^H {\bf S}_\kappa{\bf z}$.
\end{lemma}
\begin{proof}
We make use of arguments from \cite{book_Elman}. Pick ${\bf z}\in \mathbb{C}^{|\Phi|}$ and define $\mathbbm{z}^\delta = \sum \varphi_i {\bf z}_i$. The inequalities in \eqref{chap4:eq:conditioning Schur} follow from
\begin{align}\label{chap4:eq:lemma schur 1}
(\gamma^\delta_\kappa)^2\leq\frac{\|B_\kappa \mathbbm{z}^\delta \|^2_{{V_{\kappa}^\delta}'}}{\|\mathbbm{z}^\delta\|^2_U} \leq 1,
\end{align}
and
\begin{align*}
    {\bf z}^H{\bf S}_\kappa {\bf z} &= 
(({\bf M}^{V_\kappa})^{-\tfrac12} {\bf B}_\kappa{\bf z})^H(({\bf M}^{V_\kappa})^{-\tfrac12} {\bf B}_\kappa{\bf z}) \\
&=\sup_{{\bf w}\not = 0}\frac{({\bf w}^H({\bf M}^{V_\kappa})^{-\tfrac12} {\bf B}_\kappa{\bf z})^2}{{\bf w}^H{\bf w}} \\
&= \sup_{{\bf v}\not=0}\frac{({\bf v}^H {\bf B}_\kappa {\bf z})^2}{{\bf v}^H {\bf M}^{V_\kappa}{\bf v}} \\
&= \|B_\kappa \mathbbm{z}^\delta \|^2_{{V_\kappa^\delta}'},
\end{align*}
which implies that
\begin{align}\label{chap4:eq:lemma schur 2}
\frac{{\bf z}^H{\bf S}_\kappa {\bf z}}{{\bf z}^H{\bf M}^U{\bf z}} = \frac{\|B_\kappa \mathbbm{z}^\delta \|^2_{{V_\kappa^\delta}'}}{\|\mathbbm{z}^\delta\|^2_U}.
\end{align} 
The second statement follows from the definition of $\gamma_\kappa^\delta$ in \eqref{chap4:inf-sup definition}.
\end{proof}

For quasi-uniform meshes, it is known that after equipping $U^\delta$ with Lagrange bases, the condition number of ${\bf M}^U$ is uniformly bounded (not in $p$ however). In this case, a preconditioner equal to a suitable scalar times the identity would ensure $\gamma_{S_\kappa} \eqsim (\gamma^\delta_\kappa)^2$ and $\Gamma_{S_\kappa}\eqsim 1$, where the hidden constants depend on the shape regularity parameters of the mesh. In this work, however, we aim for the spectrum of ${\bf M}^U$ to be clustered around $1$ for \emph{any} conforming mesh, so that $\gamma_{S_\kappa} \eqsim (\gamma^\delta_\kappa)^2$ and $\Gamma_{S_\kappa}\eqsim 1$ hold when ${\bf Q}_S:=\identity$.

To achieve this, we will rescale the Lagrange basis functions according to the local mesh-size. For each element $K$, we define $h_K^d:=\vol(K)/\vol(\hat{K})$, where $\hat K$ is some reference element. Let $\tilde\Phi:=\{\tilde \varphi_i\colon i\in I\}$ be the set of Lagrange basis functions.
Then, for each $i\in I$, let $h_i:=(\sum_{K\subset \supp \phi_i} h_K^d)^{1/d}$ and set $\varphi_i = \frac{1}{h_i^d}\tilde \varphi_i$. We define our rescaled basis $\Phi:=\{\varphi_i\colon i\in I\}$. We obtain the following lemma.
\begin{lemma}
We have ${\bf u}^H{\bf M}^U {\bf u}\eqsim {\bf u}^H{\bf u}$, where the hidden constants depend solely on the conditioning of the finite element basis on the reference element $\hat{K}$.
\end{lemma}
\begin{proof}
Let $(N_i)_{i\in I}$ be the set of dual basis functions of $\tilde \Phi$, i.e. $N_i(\tilde \varphi_j) = \delta_{i,j}$, where $\delta_{i,j}$ is the Kronecker delta function. Furthermore, let $(\hat\varphi_j)_{j\in J_{\hat K}}$ and $(\hat N_j)_{j\in J_{\hat K}}$ be the reference basis on $\hat K$ and its dual basis respectively. Lastly, for $K\in \tria^\delta$, we define the set $I_K:=\{i\in I\colon K\subset \supp\varphi_i\}$.  

Let $F_K\colon \hat K\to K$ be the \new{affine} mapping which maps $\hat K$ onto $K$. For any $\mathbbm{z}^\delta\in P_p(K)$ let $\widehat{\mathbbm{z}^\delta}(\cdot):=\mathbbm{z}^\delta(F_K(\cdot))$. Thanks to the Lagrange basis being \emph{affine equivalent}, for any $\mathbbm{z}^\delta\in P_p(K)$ it holds that 
\begin{align}\label{chap4:eq:affine equivalence}
    \sum_{j\in J_{\hat K}}|\hat N_j(\widehat{\mathbbm{z}^\delta})|^2 = \sum_{i\in I_K} |N_i(\mathbbm{z}^\delta)|^2.
\end{align}
Write $\mathbbm{z}^\delta = \sum_j {\bf z}_j\phi_j$, then
    \begin{align*}
        &{\bf z}^H{\bf M}^U {\bf z} = 
        \|\mathbbm{z}^\delta\|^2_{(L_2(\Omega))^{d+1}}= 
        \sum_K\|\mathbbm{z}^\delta|_K\|_{(L_2(K))^{d+1}}^2 = 
        \sum_K h_K^d \|\widehat{\mathbbm{z}^\delta|_K}\|^2_{(L_2(\hat K))^{d+1}} \eqsim \\
        &\sum_K h_K^d \sum_{j \in J_{\hat K}} |\hat N_j(\widehat{\mathbbm{z}^\delta|_K})|^2=
        \sum_K h_K^d \sum_{i \in I_K} |N_i(\mathbbm{z}^\delta|_K)|^2 =
        \sum_{i\in I} \sum_{K\subset\supp \phi_i} h_K^d |N_i(\mathbbm{z}^\delta|_K)|^2=\\
        &\sum_{i\in I} \sum_{K\subset\supp \phi_i} \frac{h_K^d}{h_i^d} |{\bf z}_i|^2=
        {\bf z}^H{\bf z}.
    \end{align*}
    Here we used norm-equivalence on the space $P_p(\hat K)$ and \eqref{chap4:eq:affine equivalence}.
\end{proof}
\begin{remark}\label{chap4:remark:Chebyshev}
    Instead of choosing ${\bf Q}_S = \identity$, one could let $({\bf Q}_S)^{-1}$ be the result of a few Chebyshev or Richardson iterations using the matrix ${\bf M}^U$. In this way, we can get ${\bf Q}_S$ to be as close to ${\bf M}^U$ as we want. Although the application of $({\bf Q}_S)^{-1}$ then becomes more expensive, this approach reduces the number of iterations needed for the iterative solver of \eqref{chap4:discrete}, which can reduce the overall cost. 
\end{remark}

\begin{remark}
When $U^\delta = \cS_p^{-1}(\tria^\delta)^{d+1}$, then $({\bf M}^U)^{-1}$ can be applied in $\mathcal{O}(n)$ operations and we may choose ${\bf Q}_S={\bf M}^U$.
\end{remark}

\subsection{Preconditioner for ${\bf M}^{V_\kappa}$}\label{chap4:sec:preconditionerV}
For preconditioning the matrix ${\bf M}^{V_\kappa}$ we will make use of Hermitian successive subspace corrections (HSSC) in the space $V_\kappa^\delta$. The theory of successive subspace corrections is well-established (see, for example, \cite{MR1193013}), but in order to keep the discussion self-contained, we will briefly summarize some key concepts here.
Let $X$ be some Hilbert space. For successive subspace corrections, we need a sequence of subspaces $(X_i)_{i=0}^N$, where $N$ is some integer. The corresponding \emph{successive subspace correction operator} $Q^{-1}\colon X'\to X$ is defined as follows: for $f\in X'$ we define $x = Q^{-1}{f}$, where $x$ is computed using the algorithm below.
\begin{itemize}\label{chap4:correction-operator}
    \item Set $x = 0$.
    \item For $i = 0,1,\ldots, N-1, N$ let $w$ solve
    $$
    \langle w, \undertilde{w}\rangle_{X} = f(\undertilde{w}) - \langle v, \undertilde{w}\rangle_{X}\mbox{ for all } \undertilde{w}\in X_i,
    $$
    and set $v\leftarrow v + w$. 
\end{itemize}

In our case, we want the preconditioner $Q^{-1}_{V_\kappa}\colon {V_\kappa^\delta}' \to V_\kappa^\delta$ to be Hermitian positive definite. To achieve this, we choose a sequence of subspaces 
\begin{align}\label{chap4:eq:subspaces}
    V_0^\delta, V_1^\delta,\ldots, V_{N-1}^\delta, V_N^\delta, V_{N-1}^\delta\ldots, V_1^\delta, V_0^\delta\subset V_\kappa^\delta,
\end{align}
satisfying $\sum_{i =0}^N V_i^\delta = V_\kappa^\delta$. By visiting each subspace $V_i^\delta$ twice—once in forward order, once in reverse—except for $V_N^\delta$, we ensure that $Q_{V_\kappa}$ is a Hermitian operator (the second pass through $V_N^\delta$ would be redundant and is thus omitted).
The Hermitian matrix ${\bf Q}_{V_\kappa}$ is defined as the matrix representation of the mapping $Q_{V_\kappa}$, i.e. we have $({\bf Q}_{V_\kappa})_{i,j} = (Q_{V_\kappa}\psi_j)(\psi_i)$. 

To show that the preconditioner is positive definite we now show some simple facts about the spectrum of ${\bf Q}_{V_\kappa}^{-1}{\bf M}^{V_\kappa}$.
Denote by $P_i\colon V\to V^\delta_i$ the orthogonal projection operator onto $V^\delta_i$ w.r.t $\langle \cdot, \cdot\rangle_{V_\kappa}$, and define $M^{V_\kappa^\delta}\colon V_\kappa^\delta\to {V_\kappa^\delta}'$ by $(M^{V_\kappa^\delta}\mathbbm{w}^\delta)(\mathbbm{v}^\delta) = \langle \mathbbm{w}^\delta, \mathbbm{v}^\delta\rangle_{V_\kappa}$, i.e. ${\bf M}^{V_\kappa}$ is the matrix representation of $M^{V_\kappa^\delta}$.
Using induction with respect to the number of subspaces $N$, and using $(I-P_N)(I-P_N)=(I-P_N)$ one can deduce that
$$
\mathbbm{z}^\delta - Q^{-1}_{V_\kappa}M^{V_\kappa^\delta} \mathbbm{z}^\delta = \prod_{i=0}^N(I-P_i)\prod_{i=N}^0 (I-P_i)\mathbbm{z}^\delta.
$$
This implies that 
\begin{align}\label{chap4:error_ssc}
    \langle Q^{-1}_{V_\kappa}M^{V_\kappa^\delta} \mathbbm{z}^\delta, \mathbbm{z}^\delta\rangle_{V_\kappa} = \langle \mathbbm{z}^\delta, \mathbbm{z}^\delta\rangle_{V_\kappa} - \langle \prod_{i=N}^0 (I-P_i)\mathbbm{z}^\delta, \prod_{i=N}^0 (I-P_i)\mathbbm{z}^\delta\rangle_{V_\kappa} \leq  \langle \mathbbm{z}^\delta, \mathbbm{z}^\delta\rangle_{V_\kappa}.
\end{align}
Hence, $\Gamma_{V_\kappa}\leq 1$. Since, for $\mathbbm{z}^\delta\in V_0^\delta$, it holds that $Q^{-1}_{V_\kappa}M^{V_\kappa^\delta} \mathbbm{z}^\delta = \mathbbm{z}^\delta$, we conclude that $\Gamma_{V_\kappa} = 1$. 

Furthermore, $\sum_{i =0}^N V^\delta_i = V_\kappa^\delta$ implies that $\|\prod_{i=N}^0 (I-P_i)\mathbbm{z}^\delta\|_{V_\kappa} < \| \mathbbm{z}^\delta\|_{V_\kappa}$ for any $\mathbbm{z}^\delta\in V_\kappa^\delta$. Together with \eqref{chap4:error_ssc}, we can conclude that the lower bound $\gamma_{V_\kappa}>0$ holds, i.e. the preconditioner is a positive definite operator.

By the arguments above, any HSSC operator in $V_\kappa^\delta$ can be used as a preconditioner for ${\bf M}^{V_\kappa}$. To make the iterative solver more efficient, one has to choose the subspaces $(V^\delta_i)_{i=0}^N$ appropriately. This is a difficult task. On the one hand, the subspaces need to be rich enough for the spectrum of $Q^{-1}_{V_\kappa}M^{V_\kappa^\delta}$ to be small enough, but on the other hand, the subspaces need to be small for an efficient application of the preconditioner. In the next two sections, we introduce our chosen approach, which is guided by well-established principles.

\subsubsection{Multigrid}\label{chap4:multigrid}
Usually, it is beneficial to include subspaces that can be represented on coarse meshes. By including subspaces represented on multiple meshes, we are able to effectively smoothen high-frequency components on fine meshes and low-frequency components on coarse meshes. 

To efficiently perform corrections on subspaces represented on coarse meshes we employ multigrid operators.

Let $\tria^\delta_0 \prec\tria^\delta_1\prec\ldots\prec\tria^\delta_L=\tria^\delta$ be some nested sequence of triangulations. For each triangulation $\tria^\delta_\ell$, let $V_\ell =(\cS_{\tilde{p}}^0(\tria^\delta_\ell) \times \RT_{\tilde{p}}(\tria^\delta_\ell)) \cap V_\kappa^\delta$ be the corresponding finite element subspace of $V_\kappa^\delta$. 

For $\ell=1,\ldots, L$, let $I_\ell\colon V_{\ell-1}\to V_\ell$ be the inclusion operators and let the dual mapping ${I'_\ell}\colon V'_\ell\to V'_{\ell-1}$ be defined by $({I'_\ell}\mathbbm{f})(\mathbbm{v}) = \mathbbm{f}(I_\ell \mathbbm{v})$. Furthermore, on each level, we define $M_\ell\colon V_\ell\to V_\ell'$ as $(M_\ell \mathbbm{v})(\mathbbm{w}) = \langle \mathbbm{v}, \mathbbm{w} \rangle_{V_\kappa}$.
Finally, on each level we use successive subspace correction operators $S^{-1}_\ell \colon V_\ell'\to V_\ell$ and $(S^*_\ell)^{-1} \colon V_\ell'\to V_\ell$, which are called \emph{smoothers}. For $(V_{\ell, i})_{i=0}^{N_\ell}\subset V_\ell$ being the sequence of subspaces that define $S^{-1}_\ell$, the operator $(S^*_\ell)^{-1}$ is defined using the same sequence of subspaces, but in reversed order.

The variable V-cycle operator $Q^{-1}_\ell\colon V_\ell'\to V_\ell$ is then defined by induction as follows. Setting $Q^{-1}_{-1} = 0$, assuming that $Q^{-1}_{\ell-1}$ has been defined, for $\mathbbm{f}\in V_\ell'$ we define $\mathbbm{v}^\delta = Q^{-1}_\ell \mathbbm{f}$ by the following:
\begin{itemize}
    \item Set $\mathbbm{v}^\delta = 0$.
    \item For $k = 1,\ldots, m_\ell$, set
    $$
        \mathbbm{v}^\delta \leftarrow \mathbbm{v}^\delta + S^{-1}_\ell(\mathbbm{f}-M_\ell \mathbbm{v}^\delta).
    $$
    \item Set $\mathbbm{v}^\delta \rightarrow \mathbbm{v}^\delta + I_\ell Q^{-1}_{\ell-1}I_\ell'(\mathbbm{f}-M_\ell \mathbbm{v}^\delta).$
    \item For $k = 1,\ldots, m_\ell$, set
    $$
    \mathbbm{v}^\delta \leftarrow \mathbbm{v}^\delta + (S^*_\ell)^{-1}(\mathbbm{f}-M_\ell \mathbbm{v}^\delta).
    $$
\end{itemize}
Note that $S_\ell$ does not need to be Hermitian since we visit the subspaces that define $S^{-1}_\ell$ in reversed order when applying $(S^*_\ell)^{-1}$.

In the numerical experiments, we will choose $m_\ell = 1$ for all levels $\ell$, unless stated otherwise.

To see that the multigrid operator defined above is in fact an HSSC operator, we can use an argument by induction. If we assume that $Q_{\ell-1}$ is an HSSC operator in the space $V_{\ell-1}$ with the sequence of subspaces $(W_i^\delta)_{i=0}^N$, it follows that $I_\ell Q^{-1}_{\ell-1}I_\ell'$ is also an HSSC operator in the space $V_\ell$, but with subspaces $(I_\ell W_i^\delta)_{i=0}^N$. Then it is easy to conclude that $Q_\ell$ is an HSSC operator in the space $V_\ell$. Of course, to ensure $Q_\ell$ being positive definite, we need the sum of all the subspaces we encounter to be equal to $V_\ell$.

\subsubsection{Choice of smoother}
The choice of our smoother $S_\ell$ is based on results on multigrid operators for $H(\divv;\Omega)$ and $H^1(\Omega)$. For $H(\divv;\Omega)$, it is known that using a successive subspace correction operator with subspaces defined on vertex patches as a smoother leads to an efficient multigrid V-cycle preconditioner \cite{14.24, 14.26}. In contrast, for multigrid methods for $H^1(\Omega)$, a simple Gauss-Seidel smoother, where each subspace consists of the span of only a single function, is already sufficient. Hence, a smoother based on vertex patches would already give rise to a uniform preconditioner for bounded values of $\kappa$, since then it holds that ${\|\cdot\|_{V_{\kappa}}}\eqsim {\|\cdot\|_{H^1(\Omega)\times H(\divv;\Omega)}}$. Motivated by the above we will define our smoothers using function spaces on vertex patches.

For each vertex $\nu\in \cN_\ell$, let $V^\nu_\ell = \{ \mathbbm{v}^\delta\in V_\ell \colon \supp \mathbbm{v}^\delta \subseteq \omega^\nu_\ell \}$, be the subspace of finite element functions supported on the vertex patch $\omega^\nu_\ell:=\{K\in \mathcal{T}_\ell\colon \nu\in \overline{K}\}$. Recalling that we generally allow locally refined meshes, let $\nu_0,\ldots, \nu_{N_\ell}$ be a numbering of the vertices in $\tria^\delta_\ell$ for which the function space $V^{\nu_i}_\ell$ is not included in $V_{\ell-1}$.
By excluding vertex patches that have not been refined relative to the previous mesh, we ensure that the preconditioner can be implemented with a computational complexity of $\mathcal{O}(n)$, where $n$ denotes the number of degrees of freedom.

The smoother $S_\ell$ is defined as follows: for $\mathbbm{f} \in V_\ell'$ we define $\mathbbm{v}^\delta = S_\ell^{-1}\mathbbm{f}$ by the following
\begin{itemize}
    \item Set $\mathbbm{v}^\delta = 0$.
    \item For $\nu_j\in \cN_\ell$, $j = 1,2,\ldots,  N_\ell$ let
    $$
    \mathbbm{v}^\delta\leftarrow \mathbbm{v} + \mathbbm{w}^\delta,
    $$
    where $\mathbbm{w}^\delta\in V^{\nu_j}_\ell$ solves
    $$
    \langle \mathbbm{w}^\delta, \tilde{\mathbbm{w}}^\delta\rangle_{V_\kappa} = \mathbbm{f}(\tilde{\mathbbm{w}}^\delta) - \langle \mathbbm{v}^\delta, \tilde{\mathbbm{w}}^\delta\rangle_{V_\kappa}, \mbox{ for all } \tilde{\mathbbm{w}}^\delta\in V_\ell^{\nu_j}.
    $$    
\end{itemize}

\begin{remark}[Static condensation]
If $\tilde p$ is relatively large, one may employ static condensation at each level to enable a more efficient application of the above smoother $S_\ell$. Given the basis $\Psi$ of $V_\ell$ as described in Remark~\ref{chap4:remark:robin-discrete}, we can partition it into two disjoint sets of basis functions, $\Psi^\mathcal{K}$ and $\Psi^\mathcal{S}$, with $\Psi^\mathcal{K} \dot\cup \Psi^\mathcal{S} = \Psi$, such that each function in $\Psi^\mathcal{K}$ is supported on a single element $K \in \tria_\ell$.

We then define the subspaces $V_\ell^\mathcal{K} := \operatorname{span} \Psi^\mathcal{K}$ and $V_\ell^\mathcal{S} := \operatorname{span} \Psi^\mathcal{S}$. Next, consider the operator  
$$
M_\ell\colon V_\ell^\mathcal{K} \times V_\ell^\mathcal{S} \to (V_\ell^\mathcal{K})' \times (V_\ell^\mathcal{S})',
$$ 
defined by  
$$
M_\ell(\phi_\mathcal{K}, \phi_\mathcal{S})(\undertilde{\phi_\mathcal{K}}, \undertilde{\phi_\mathcal{S}}) = \langle \phi_\mathcal{K} + \phi_\mathcal{S}, \undertilde{\phi_\mathcal{K}} + \undertilde{\phi_\mathcal{S}} \rangle_{V_\kappa}.
$$ 
In block-form, this can be written as  
$$
M_\ell = \begin{bmatrix}
    M_{\mathcal{K}, \mathcal{K}} & M_{\mathcal{K}, \mathcal{S}} \\
    M_{\mathcal{S}, \mathcal{K}} & M_{\mathcal{S}, \mathcal{S}}
\end{bmatrix},
$$ 
where $M_{\mathcal{A},\mathcal{B}} \colon V_\ell^\mathcal{B} \to (V_\ell^\mathcal{A})'$ is defined by  
$$
(M_{\mathcal{A},\mathcal{B}} \psi_\mathcal{B})(\psi_\mathcal{A}) = \langle \psi_\mathcal{A}, \psi_\mathcal{B} \rangle_{V_\kappa},\quad \psi_\mathcal{A}\in V_\ell^\mathcal{A}, \psi_\mathcal{B}\in V_\ell^\mathcal{B},
$$ 
for $\mathcal{A}, \mathcal{B} \in \{\mathcal{K}, \mathcal{S}\}$.

Furthermore, define the transformation  
$$
\Xi = \begin{bmatrix}
    I & -M_{\mathcal{K},\mathcal{K}}^{-1} M_{\mathcal{K},\mathcal{S}} \\
    0 & I
\end{bmatrix}.
$$ 
It follows that for any $\psi_\mathcal{S} \in V_\ell^\mathcal{S}$ and $\psi_\mathcal{K} \in V_\ell^\mathcal{K}$, we have  
$$
M_\ell\left( \Xi \begin{Bmatrix} 0 \\ \psi_\mathcal{S} \end{Bmatrix} \right)\left( \begin{Bmatrix} \psi_\mathcal{K} \\ 0 \end{Bmatrix} \right) = 0.
$$ 
In other words, the space $V_\ell$ decomposes orthogonally with respect to the $V_\kappa$-inner product as  
\begin{align}\label{chap4:eq:orthogonal subspaces}
    V_\ell = V_\ell^\mathcal{K} \oplus^{\perp_{V_\kappa}} \Xi (V_\ell^\mathcal{S}).
\end{align}
For each vertex $\nu \in \mathcal{N}_\ell$, define the local subspaces  
\begin{align}\label{chap4:eq:vertexpatch_edge}
    \tilde{V}_\ell^\nu := \{ \mathbbm{v}^\delta \in \Xi(V_\ell^\mathcal{S}) \colon \operatorname{supp}(\mathbbm{v}^\delta) \subseteq \omega_\ell^\nu \},
\end{align}
and 
\begin{align}\label{chap4:eq:vertexpatch_el}
V_\ell^{\nu,\mathcal{K}}:=\{\mathbbm{v}^\delta\in V_\ell^\mathcal{K}\colon \supp(\mathbbm{v}^\delta)\subseteq \omega_\ell^\nu\}.
\end{align}
It follows from~\eqref{chap4:eq:orthogonal subspaces} and $V_\ell^\nu = \Span \tilde{V}_\ell^\nu\dot\cup V_\ell^{\nu,\mathcal{K}}$, that for every vertex $\nu$, subspace corrections in $V_\ell^\nu$ can equivalently be performed by first applying a subspace correction in $V_\ell^{\nu,\mathcal{K}}$, followed by a correction in $\tilde{V}_\ell^\nu$. 

Furthermore, since $V_\ell^{\nu,\mathcal{K}}$ is orthogonal to all other subspaces of the form \eqref{chap4:eq:vertexpatch_edge} or \eqref{chap4:eq:vertexpatch_el},
it follows that the subspace correction in $V_\ell^{\nu,\mathcal{K}}$ can be postponed until all the remaining subspace corrections have been performed.
This reasoning can be applied to all vertex patches, to conclude that the smoother $S_\ell$ can equivalently be defined by $\mathbbm{v}^\delta = S_\ell^{-1} \mathbbm{f}$, where $\mathbbm{v}^\delta$ is computed as follows:

\begin{itemize}
    \item Initialize $\mathbbm{v}^\delta = 0$.
    \item For each vertex $\nu_j \in \mathcal{N}_\ell$, $j = 1, \dots, \#\mathcal{N}_\ell$, update
    $$
    \mathbbm{v}^\delta \leftarrow \mathbbm{v}^\delta + \mathbbm{w}^\delta,
    $$ 
    where $\mathbbm{w}^\delta \in \tilde{V}_\ell^{\nu_j}$ solves
    $$
    \langle \mathbbm{w}^\delta, \tilde{\mathbbm{w}}^\delta \rangle_{V_\kappa} = \mathbbm{f}(\tilde{\mathbbm{w}}^\delta) - \langle \mathbbm{v}^\delta, \tilde{\mathbbm{w}}^\delta \rangle_{V_\kappa}, \quad \forall \tilde{\mathbbm{w}}^\delta \in \tilde{V}_\ell^{\nu_j}.
    $$
    \item Finally, update
    $$
    \mathbbm{v}^\delta \leftarrow \mathbbm{v}^\delta + \mathbbm{w}^\delta,
    $$ 
    where $\mathbbm{w}^\delta \in V_\ell^\mathcal{K}$ solves
    $$
    \langle \mathbbm{w}^\delta, \tilde{\mathbbm{w}}^\delta \rangle_{V_\kappa} = \mathbbm{f}(\tilde{\mathbbm{w}}^\delta) - \langle \mathbbm{v}^\delta, \tilde{\mathbbm{w}}^\delta \rangle_{V_\kappa}, \quad \forall \tilde{\mathbbm{w}}^\delta \in V_\ell^\mathcal{K}.
    $$
\end{itemize}
Given two basis functions $\psi^\mathcal{S}, \undertilde{\psi}^\mathcal{S} \in V_\ell^\mathcal{S}$, their inner product under the transformation $\Xi$ reads  
$$
\langle \Xi \psi^\mathcal{S}, \Xi \undertilde{\psi}^\mathcal{S} \rangle_{V_\kappa} = (\mathcal{S} \psi^\mathcal{S})(\undertilde{\psi}^\mathcal{S}),
$$ 
where  
$$
\mathcal{S} := M_{\mathcal{S}, \mathcal{S}} - M_{\mathcal{S}, \mathcal{K}} M_{\mathcal{K}, \mathcal{K}}^{-1} M_{\mathcal{K}, \mathcal{S}}.
$$ 
Computing the matrix representation of $\mathcal{S}$ is not computationally expensive since $M_{\mathcal{K}, \mathcal{K}}$ is block-diagonal and can be inverted efficiently.
For the same reason, the subspace correction in $V_\ell^\mathcal{K}$ is computationally efficient.

Since each local subspace $\tilde{V}_\ell^{\nu_j}$ is smaller than the corresponding $V_\ell^{\nu_j}$, the associated corrections are cheaper to compute, making this implementation of the smoother $S_\ell$ more efficient when $\tilde{p}$ is sufficiently large.
\end{remark}

\begin{remark}\label{chap4:remark:Helmholtz-kernel}
The intuition behind different aspects of our preconditioner can be understood through the interaction of different error components with the $V_\kappa$-inner product. Recall that the inner product on $V_\kappa$ is given by 
$$
\langle (\eta, \vec v),(\undertilde{\eta}, \undertilde{\vec v})\rangle_{V_\kappa} = \langle -\eta -\frac{1}{\kappa}\divv \vec v, -\undertilde{\eta} -\frac{1}{\kappa}\divv \undertilde{\vec v}\rangle_{L_2(\Omega)}
+ \langle \frac{1}{\kappa}\nabla \eta - \vec v, \frac{1}{\kappa}\nabla \undertilde{\eta} - \undertilde{\vec v}\rangle_{L_2(\Omega)^d}.
$$

First, we can argue that coarse subspaces are essential for resolving smooth error components. Assume that the mesh-size satisfies $h\ll\frac{1}{\kappa}$. Let $(\varphi, \psi)\in V_\kappa^\delta$ where $\varphi$ is a Lagrange basis function and $\psi$ is a Raviart-Thomas basis function that is not divergence-free, and let $(\eta, \vec{v})\in V^\delta_\kappa$ be a smooth, non-oscillatory function. For simplicity we assume that $(\eta, \vec v)=(1,\vec{1})$.

Standard scaling arguments show that 
$\|\psi\|_{L_2(\Omega)^d} \eqsim\|\varphi\|_{L^2(\Omega)} \eqsim h^{d/2}$, while \( \|\frac{1}{\kappa}\divv\psi\|_{L_2(\Omega)}\eqsim \|\frac1\kappa\nabla\varphi\|_{L_2(\Omega)} \eqsim h^{\frac{d-2}{2}}/\kappa \). Hence, on fine meshes, it holds that 
$\|(\varphi,\psi)\|_{V_\kappa}\approx h^{\frac{d-2}{2}}/\kappa $ and $|\langle (\eta, \vec v), (\varphi, \psi)\rangle_{V_\kappa}|\approx h^d.$
Performing a correction on a subspace spanned by \( (\varphi,\psi) \), by solving \( \langle c( \varphi, \psi), (\varphi,\psi) \rangle_{V_\kappa} = \langle (\eta, \vec v), (\varphi, \psi) \rangle_{V_\kappa} \), therefore yields \( |c| \approx \kappa^2 h^2 \). The resulting correction \( c \cdot(\varphi,\psi) \) is very small and does not locally resemble \( (\eta, \vec v) \) well. Furthermore, there are not enough divergence-free Raviart-Thomas basis functions to avoid this problem. Consequently, a smoother defined on a fine mesh has little impact on smooth components. On coarser meshes, however, smoothing does have an effect, which justifies the inclusion of coarse levels in the preconditioner.

The use of subspaces defined on vertex patches is motivated by the behavior of functions in \( H(\operatorname{div}; \Omega) \). 
Suppose \( (0, \vec{u}) \in V^\delta_\kappa\), where $\vec{u}$ is an oscillatory, divergence-free function.
Once again, through similar reasoning as above, performing corrections on subspaces spanned by one Raviart-Thomas basis function has very little effect on a fine mesh since there are not enough divergence-free Raviart-Thomas basis functions. Furthermore, since the function $\vec{u}$ is not represented on coarse meshes, corrections on coarse meshes will also not be effective.
However, on vertex patches one can construct additional discrete functions in $\RT_p(\tria^\delta)$ that are divergence-free, enabling more effective local corrections. This justifies the incorporation of vertex-patch subspaces.

Finally, in our context we also encounter \( (\eta, \vec{v}) \in V^\delta_\kappa\) satisfying
\[
( -\eta - \frac{1}{\kappa} \operatorname{div} \vec{v}, \, \frac{1}{\kappa} \nabla\eta - \vec{v} ) \approx (0, 0).
\]
Such functions approximately satisfy the equation \( \triangle \eta + \kappa^2 \eta = 0 \), and can be viewed as lying in the 'Helmholtz kernel'. Since these functions are typically highly oscillatory, coarse-grid corrections have limited effect on them. Moreover, vertex-patch subspaces do not contain functions of this type. As a result, these components are only weakly affected by the preconditioner. This explains why the preconditioner does not yield condition numbers that are uniformly bounded with respect to $\kappa$.
\end{remark}

\subsection{Stopping criterion}\label{chap4:section:stoppingcriteria}

    An advantage of iterative methods is the ability to stop the solution process as soon as the algebraic errors become insignificant compared to the total errors. Beyond this point, performing additional iterations of the matrix-vector solver does not increase the quality of the finite element approximation. 
    
    In this section, we develop heuristics to approximate the \emph{total error} and the \emph{algebraic error} when we use the MINRES algorithm as the iterative solver. Based on these approximations, we propose a stopping criterion: terminate the MINRES algorithm when the \emph{approximate algebraic error} is a fraction of the \emph{approximate total error}. In the numerical section below, this fraction is set to $\tfrac{1}{2}$.
    
    \subsubsection{Total error}
    Following Theorem \ref{chap4:thm:boosted} the total error $\|\mathbbm{u} - \mathbbm{u}^\delta\|_U$ is bounded from below by $\|B_\kappa '\mathbbm{v}^\delta\|_U$. In fact, numerical experiments suggest that this error estimator is close to being exact on meshes where the solution is being resolved. Therefore, for intermediate solutions $(\tilde{\mathbbm{v}}^\delta, \tilde{\mathbbm{u}}^\delta)$ arising from the iterative solver, we choose $\|B_\kappa '\tilde{\mathbbm{v}}^\delta\|_{U}$ as the approximate total error.
    
    While this approximation may not be accurate in the early iterations, it remains a safe choice. Our stopping criterion does not terminate the MINRES method prematurely, as we observed that $\|B_\kappa '\tilde{\mathbbm{v}}^\delta\|_{U}$ converges to $\|B_\kappa '\mathbbm{v}^\delta\|_{U}$ with a sufficient degree of accuracy, well before convergence of $\tilde{\mathbbm{u}}^\delta$ to $\mathbbm{u}^\delta$ (see Figure \ref{chap4:fig:error_against_iteration} in Section \ref{chap4:sec:numerics}).
    \subsubsection{Algebraic error}
    The algebraic error $\|\mathbbm{u}^\delta - \tilde{\mathbbm{u}}^\delta\|_U$ can be estimated by the residual norm of the matrix-vector equation divided by some constant $c$, which depends on the quantities in \eqref{chap4:spectral-equivalence} (see, for example, \cite[Theorem 4.10]{book_Elman}). 
    From the previous sections, we already know that $\Gamma_{S_\kappa}\eqsim 1$ and $\Gamma_{V_\kappa}=1$ and $\gamma_{S_\kappa}\eqsim (\gamma_\kappa^\delta)^2$ and it only remains to find an approximation of $\gamma_{V_\kappa}$ and $\gamma_\kappa^\delta$. To simplify the situation however, we assume\footnote[2]{This is not too far away from the truth because we want to ensure that $\gamma_\kappa^\delta$ is close to $1$ in order to avoid the pollution effect. This is in turn achieved by choosing $\tilde p$ big enough.} that $\gamma_{\kappa}^\delta=\Gamma_{S_\kappa}=1$, i.e. the preconditioner for the Schur complement satisfies ${\bf Q}_S = {\bf S}_\kappa$. \new{Under these assumptions, for all negative eigenvalues $\lambda<0$ of the preconditioned saddle point system
\begin{align}
\begin{pmatrix}
{\bf M}^{V_{\kappa}} & {\bf B}_\kappa \\
{\bf B}_\kappa^H & {\bf 0}
\end{pmatrix}
\begin{pmatrix} {\bf v} \\ {\bf u} \end{pmatrix}
= \lambda
\begin{pmatrix} {\bf Q}_{V_\kappa}{\bf v} \\ {\bf Q}_S {\bf u} \end{pmatrix},
\end{align}
it follows from \cite[Theorem 4.7]{book_Elman} that
    \begin{align}
        \gamma_{V_\kappa} \leq \frac{\lambda^2}{1+\lambda}.
    \end{align}}
We will assume that for \( \bar\lambda \) being the largest negative eigenvalue of the preconditioned saddle point system, the value \( \tfrac{ \bar\lambda^2 }{ 1 + \bar\lambda } \) provides a good approximation of \( \gamma_{V_\kappa} \).

To estimate the negative eigenvalue of the preconditioned saddle point system $\bar \lambda$, we use an algorithm developed in \cite{MR2774834}. There, a practical MINRES implementation can be found where the so-called \emph{harmonic Ritz values} of the preconditioned system are computed on the fly at very low cost. It is mentioned there that the largest negative harmonic Ritz value approximates the largest negative eigenvalue $\bar\lambda$ from below (and the smallest positive harmonic Ritz value approximates the smallest positive eigenvalue from above) and that the convergence of the MINRES method is numerically observed to be related to the quality of these approximations of the smallest positive and largest negative eigenvalues. 

Motivated by the above, in each MINRES iteration we approximate $\gamma_{V_\kappa}$ by $\bar\gamma_{V_\kappa}:=\tfrac{\tilde\lambda^2}{1+\tilde\lambda}$, where $\tilde\lambda$ is the largest negative harmonic Ritz value. The resulting approximation of the algebraic error is given by $1/c$ times the norm of the residual of the matrix-vector equation, where $c^2 = \bar\gamma_{V_\kappa}\left( 1+\tfrac{1}{2\bar\gamma_{V_\kappa}} - \sqrt{1+\tfrac{1}{4(\bar\gamma_{V_\kappa})^2} }\right)\approx \bar\gamma_{V_\kappa}$ is chosen in accordance with \cite[Theorem 4.10]{book_Elman}.

\section{Numerical experiments} \label{chap4:sec:numerics}

In this section we elaborate on the numerical experiments performed in \cite{204.18}. For the Helmholtz problems considered there we further investigate the dependence of $\gamma_\kappa^\delta$ on the wave number $\kappa$ and polynomial order $p$. We analyze the preconditioner on the space $V_\kappa$, by considering the two-grid method and the multigrid method. Finally, we solve the Helmholtz problem on these domains, also for higher wave numbers than considered in \cite{204.18}. 

\subsection{Three Helmholtz problems}

In our numerical experiments, we consider three examples where $d=2$. In the first example, we will prescribe the solution, whereas in the other two, the exact solution will be unknown. We will refer to these examples in the following sections. We use newest vertex bisection for mesh refinement \cite{MR2353951}.

The examples are described by specifying the domain $\Omega$ and its boundaries $\Gamma_D, \Gamma_N, \Gamma_R$ and the data $f, g_D, g_N$ and $g_R$. We will later refer to so-called \emph{plane-wave solutions} of the form $\phi_{\kappa \vec{r}}(\vec{x}):= e^{-i\kappa\vec{r}\cdot \vec{x}}$, where $|\vec{r}| = 1$. 

The second and third examples are so-called \emph{scattering problems}, which are challenging (see \cite{38.25}).

For each problem we start with an initial triangulation $\mathcal{T}_0$ of $\overline{\Omega}$ with an assignment of the newest vertices that satisfies the so-called \emph{matching condition}. Starting from this initial mesh we create a sequence of meshes $(\mathcal{T}_n)_{n\in\mathbbm{N}}$ of $\overline{\Omega}$, where each triangulation is created from its predecessor using newest vertex bisection. This sequence of triangulations is used to define the multigrid preconditioner from Section \ref{chap4:sec:preconditionerV}.

On each mesh $\mathcal{T}_k$, we consider the trial space $U^\delta=\cS_p^{0}(\tria_k) \times \cS_p^{0}(\tria_k)^{2}$, and test space
$V_\kappa^\delta=(\cS_{\tilde{p}}^{0}(\tria_k) \times \RT_{\tilde{p}}(\tria_k)) \cap V_\kappa$ for $p = 3$ and $\tilde{p}\geq p+2$.

\subsubsection{Plane-wave solution on unit square}\label{chap4:example1}
For the first example, we consider the unit square $\Omega = (0,1)^2$ with Robin boundary conditions. For $\vec{r} = (\cos(\pi/3), \sin(\pi/3))$ and different values for $\kappa>0$, we prescribe the plane-wave solution $\phi_{\kappa\vec{r}}$ and choose the data accordingly.

In this first example, we consider a uniform sequence of meshes $(\mathcal{T}_n)_{n\in\mathbbm{N}}$ of $\overline{\Omega}$, where $\tria_k$ is created by bisecting all triangles of $\tria_{k-1}$. This initial mesh in turn is created by cutting the domain along its diagonals and designating the interior vertex as the newest vertex in all triangles.

\subsubsection{Scattering on a non-trapping domain}\label{chap4:example2}
Secondly, we consider a so-called \emph{non-trapping scattering domain}. Namely, for 
$$
D:=\{\vec{x}\in (-1,1)^2 \colon 2|x_1|-\tfrac12<x_2<|x_1|\},
$$
let $\Omega:= (-1,1)^2\setminus \overline{D}$, $\Gamma_D:=\partial D$, and $\Gamma_R:=\partial (-1,1)^2$. We set $f=0=g_D$ and 
$g_R=\frac{1}{\kappa}(\frac{1}{\kappa} \nabla \phi_{\kappa \vec{r}}\cdot\vec{n}- i \phi_{\kappa \vec{r}})|_{\Gamma_R}=\frac{-i (\vec{r}\cdot\vec{n}+ 1)}{\kappa}\phi_{\kappa \vec{r}}|_{\Gamma_R}$, where $\vec{r}=(\cos(\pi/3),\sin(\pi/3))$.
This problem models the (soft) scattering of an incoming wave $\phi_{\kappa \vec{r}}$ by the obstacle $D$. 

The initial triangulation consist of 14 triangles where the newest vertices are chosen such that the mesh is matching. This triangulation is shown in the left picture in Figure \ref{chap4:fig:domain}.

For this example we will consider sequences of meshes generated by either uniform refinements or adaptive refinements. The adaptive refinement is driven by the a posteriori estimator presented in Section \ref{chap4:sec:errorestimation} using Dörfler marking with parameter $\theta = 0.6$.

\subsubsection{Scattering on an (elliptic) trapping domain}\label{chap4:example3}
Finally, we consider a so-called \emph{trapping domain}, where
$\Omega = (-1,1)^2\setminus \overline{D}$, $\Gamma_D=\partial D$, and $\Gamma_R=\partial (-1,1)^2$ and $D=D_1\cup D2$ where
$$
D_1=\{\vec{x}\in (-1,1)^2 \colon \tfrac{1}{4}\leq|y|\leq\tfrac{1}{4}-\tfrac{1}{2}x\}.
$$
$$
D_2=\{\vec{x}\in (-1,1)^2 \colon \tfrac{1}{2}+2x\leq |y|\leq \tfrac{1}{2}; x\geq -\tfrac{1}{2}\}.
$$
We set $f=0=g_D$ and 
$g_R=\frac{1}{\kappa}(\frac{1}{\kappa} \nabla \phi_{\kappa \vec{r}}\cdot\vec{n}- i \phi_{\kappa \vec{r}})|_{\Gamma_R}=\frac{-i (\vec{r}\cdot\vec{n}+ 1)}{\kappa}\phi_{\kappa \vec{r}}|_{\Gamma_R}$, where $\vec{r}=(\cos(9\pi/10),\sin(9\pi/10))$. 

The domain, and the initial triangulation are illustrated in Figure \ref{chap4:fig:domain}. Again, we consider both uniform and adaptive refinement strategies.

\begin{figure}[h!]
    \begin{subfigure}{.5\textwidth}
    \centering
    \includegraphics[width=\linewidth]{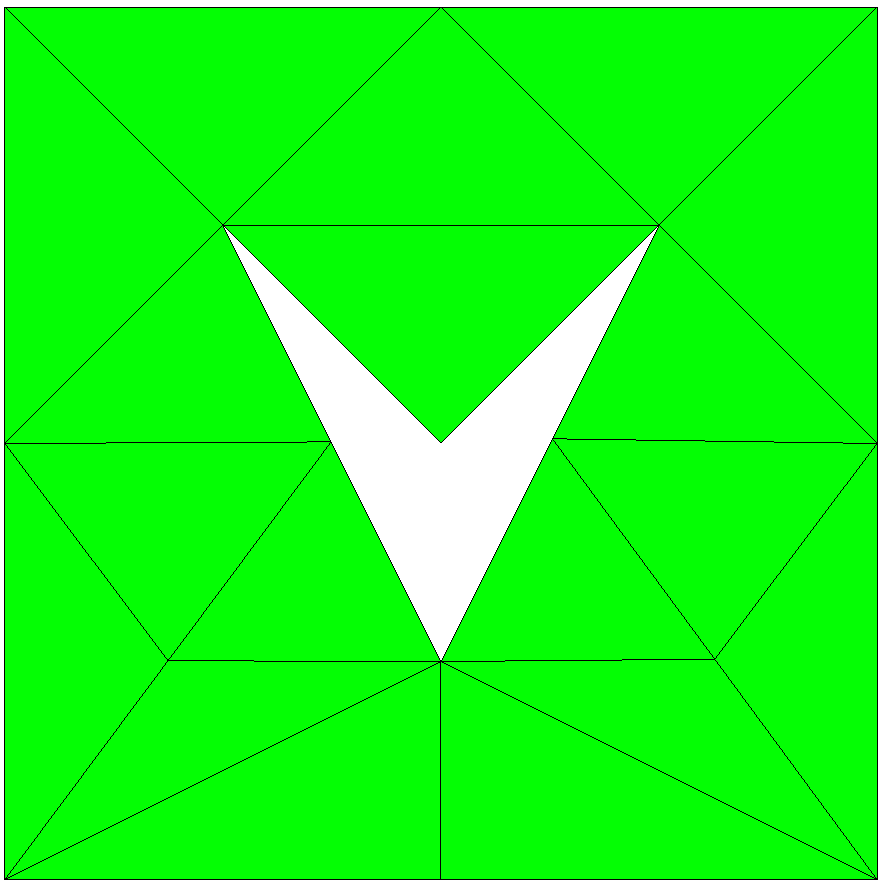}
    \end{subfigure}\hspace*{0cm}%
    \begin{subfigure}{.5\textwidth}
    \centering
    \includegraphics[width=\linewidth]{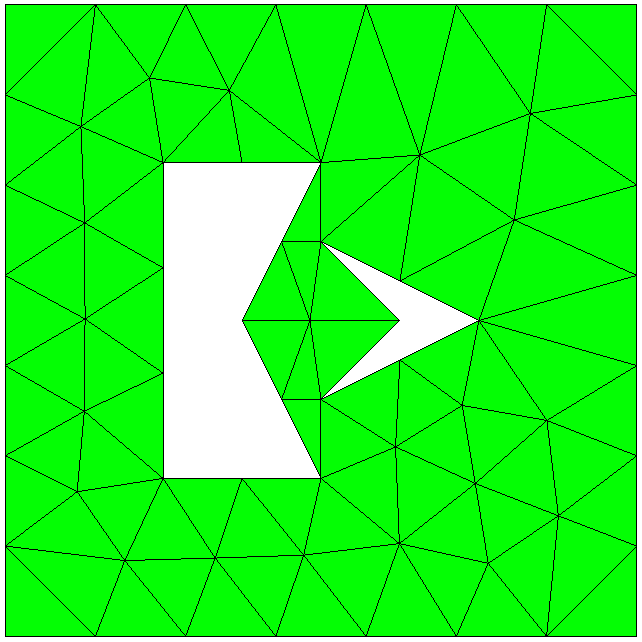}
    \end{subfigure}
\caption{Left: the non-trapping domain with its initial triangulation. Right: the trapping domain with its initial triangulation.}
\label{chap4:fig:domain}
\end{figure}

\subsection{Further investigation of pollution factors}\label{chap4:sec:numerics_pollution}

In this section we will numerically compute the inf-sup constants $\gamma_\kappa^\delta$. As mentioned before, for the first example, we know theoretically how to choose $\tilde{p}$ depending on $p$ in order to avoid pollution. But for the other examples we have to resort to numerical evidence. 

The importance of the pollution factor should not be underestimated, since both the numerical accuracy and the conditioning of the Schur complement depend on it. In addition, the quality of the error estimator deteriorates if $\tilde{p}$ is chosen too small. However, choosing $\tilde{p}$ unnecessarily large is not desirable from an efficiency point of view. 

In Figures \ref{chap4:fig:example1_inf-sup}, \ref{chap4:fig:example2_inf-sup} and \ref{chap4:fig:example3_inf-sup} we plot the inf-sup constants against the number of degrees of freedom (DoFs) in $U^\delta$ for the three examples from the previous section, where we used uniform refinement. 
For fixed $\tilde p$, the maximal pollution factor appears to grow linearly as a function of $\kappa$.

For the example from Section \ref{chap4:example3} we observe a sudden growth at $\kappa = 300$. Fortunately, not only in this case, we see that it is not difficult to produce uniformly bounded pollution factors by slightly increasing $\tilde{p}$.
\begin{figure}[!htb]
    \begin{subfigure}{.5\textwidth}
    \centering
    \includegraphics[width=\linewidth]{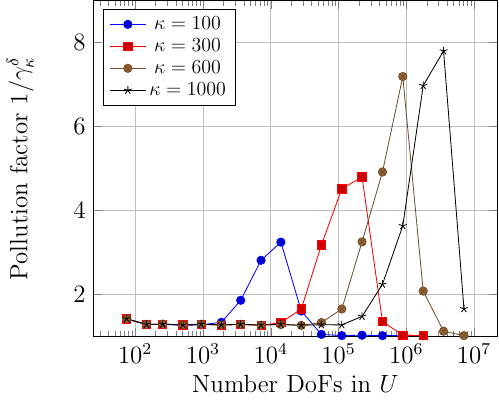}
    \end{subfigure}\hspace*{0cm}%
    \begin{subfigure}{.5\textwidth}
    \centering
    \includegraphics[width=\linewidth]{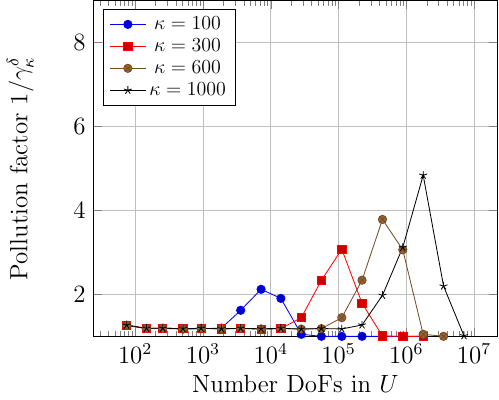}
    \end{subfigure}
\caption{Pollution factors for the example from Section \ref{chap4:example1}, for different values of $\kappa$ and $p=3$. Left: $\tilde{p}=5$, right: $\tilde{p}=6$}
\label{chap4:fig:example1_inf-sup}
\end{figure}

\begin{figure}[!htb]
    \begin{subfigure}{.5\textwidth}
    \centering
    \includegraphics[width=\linewidth]{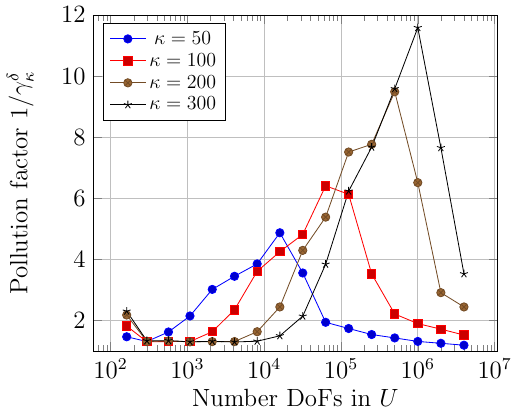}
    \end{subfigure}\hspace*{0cm}%
    \begin{subfigure}{.5\textwidth}
    \centering
    \includegraphics[width=\linewidth]{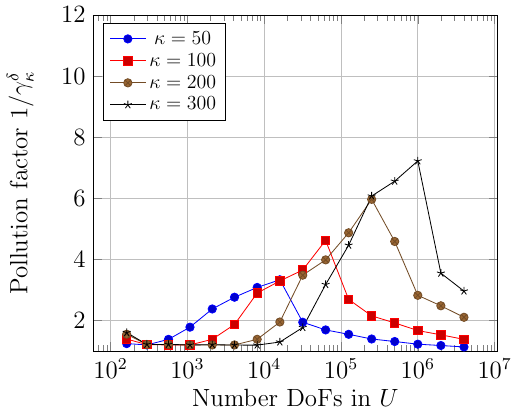}
    \end{subfigure}
\caption{Pollution factors for the example from Section \ref{chap4:example2}, for different values of $\kappa$ and $p=3$. Left: $\tilde{p}=5$, right: $\tilde{p}=6$}
\label{chap4:fig:example2_inf-sup}
\end{figure}

\begin{figure}[!htb]
    \begin{subfigure}{.5\textwidth}
    \centering
    \includegraphics[width=\linewidth]{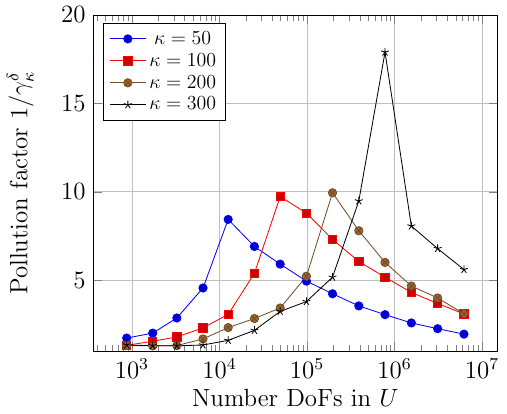}
    \end{subfigure}\hspace*{0cm}%
    \begin{subfigure}{.5\textwidth}
    \centering
    \includegraphics[width=\linewidth]{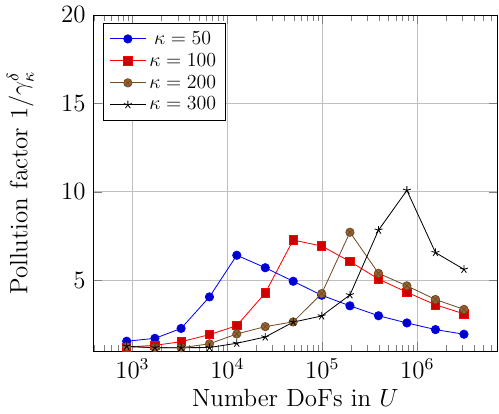}
    \end{subfigure}
\caption{Pollution factors for the example from Section \ref{chap4:example3}, for different values of $\kappa$ and $p=3$. Left: $\tilde{p}=5$, right: $\tilde{p}=6$}
\label{chap4:fig:example3_inf-sup}
\end{figure}

\FloatBarrier

\subsection{Analyzing the preconditioner}

In this section we investigate the quality of our preconditioner for $V_\kappa$. In Section \ref{chap4:sec:Iterative solvers} we mentioned that the quality of a preconditioner depends on the constants $\gamma_{V_\kappa}$ and $\Gamma_{V_\kappa}$ as defined in \eqref{chap4:spectral-equivalence}. For multiplicative subspace corrections the upper bound $\Gamma_{V_\kappa}$ is always equal to $1$ and hence the quality of the preconditioners only depends on $\gamma_{V_\kappa}$. The condition number of the preconditioned system is equal to $1/\gamma_{V_\kappa}$.

In Figures \ref{chap4:fig:example1 prec}, \ref{chap4:fig:example2 prec} and \ref{chap4:fig:example3 prec} we plot the condition number of the preconditioned system against the number of DoFs in $U^\delta$ for the three examples from the previous section, where we used uniform refinement. 

There appears to be an algebraic growth of the condition number as a function of the wave number $\kappa$.
Furthermore, we observe that the condition numbers are small on coarse meshes, increase sharply when $\frac{\kappa h}{\tilde p}\approx 1$, and stabilize on finer meshes. This sharp increase is less pronounced for the scattering domains.

\begin{figure}[h!]
    \begin{subfigure}{.5\textwidth}
    \centering
    \includegraphics[width=\linewidth]{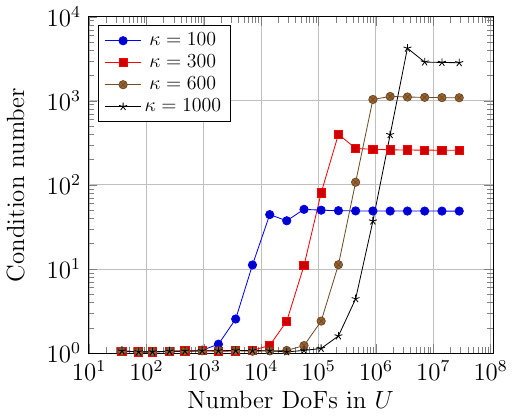}
    \end{subfigure}\hspace*{0cm}%
    \begin{subfigure}{.5\textwidth}
    \centering
    \includegraphics[width=\linewidth]{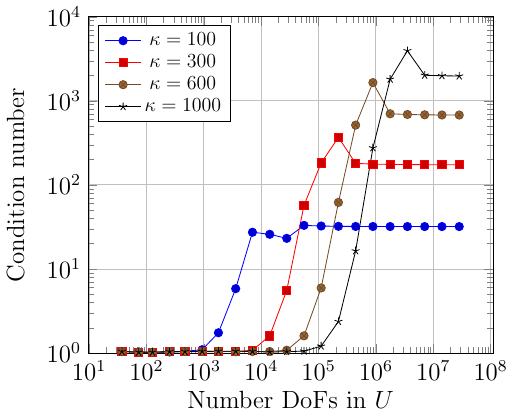}
    \end{subfigure}
\caption{Condition number of the preconditioned system for the example from Section \ref{chap4:example1}, for different values of $\kappa$. Left: $p=3$ and $\tilde{p}=5$, right: $p=3$ and $\tilde{p}=6$.}
\label{chap4:fig:example1 prec}
\end{figure}
\begin{figure}[h!]
    \begin{subfigure}{.5\textwidth}
    \centering
    \includegraphics[width=\linewidth]{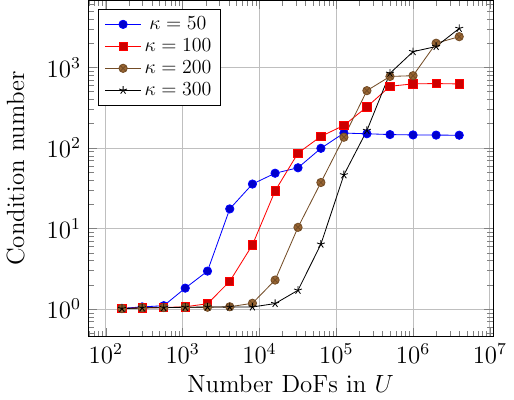}
    \end{subfigure}\hspace*{0cm}%
    \begin{subfigure}{.5\textwidth}
    \centering
    \includegraphics[width=\linewidth]{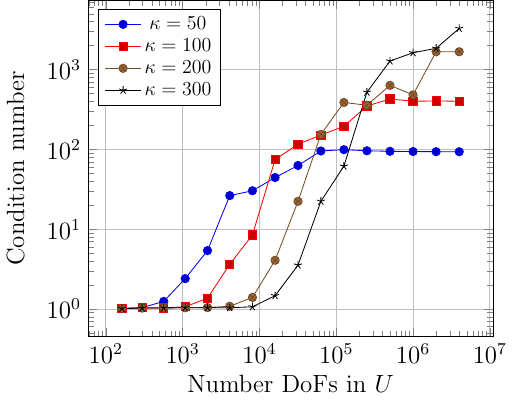}
    \end{subfigure}
\caption{Condition number of the preconditioned system for the example from Section \ref{chap4:example2}, for different values of $\kappa$. Left: $p=3$ and $\tilde{p}=5$, right: $p=3$ and $\tilde{p}=6$.}
\label{chap4:fig:example2 prec}
\end{figure}
\begin{figure}[h!]
    \begin{subfigure}{.5\textwidth}
    \centering
    \includegraphics[width=\linewidth]{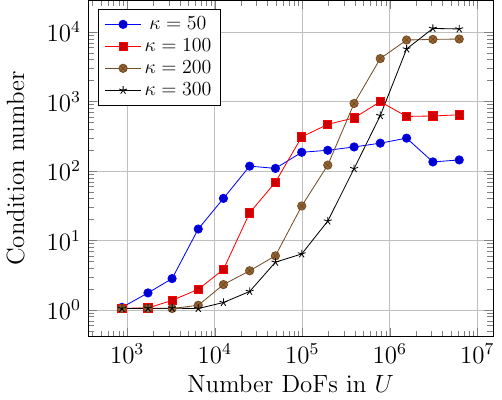}
    \end{subfigure}\hspace*{0cm}%
    \begin{subfigure}{.5\textwidth}
    \centering
    \includegraphics[width=\linewidth]{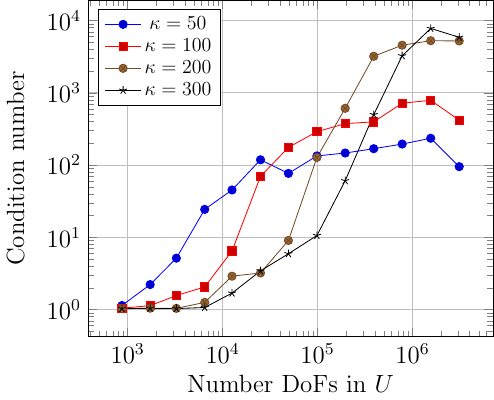}
    \end{subfigure}
\caption{Condition number of the preconditioned system for the example from Section \ref{chap4:example3}, for different values of $\kappa$. Left: $p=3$ and $\tilde{p}=5$, right: $p=3$ and $\tilde{p}=6$.}
\label{chap4:fig:example3 prec}
\end{figure}

\subsubsection{Changing $m_\ell$}\label{chap4:sec:m_ell}
Often, the behavior of the two-grid method provides useful insights about what to expect from the multigrid method. The two-grid method is obtained by replacing the sequence of meshes $\tria_0 \prec\tria_1\prec\ldots\prec\tria_L=\tria$ with the two-level hierarchy $\tria_{L-1}\prec\tria_L$ and by substituting the smoother on $\tria_{L-1}$ with $M^{-1}_{L-1}$ as defined in Section \ref{chap4:multigrid}. In Figure \ref{chap4:fig:example1_twogrid} we present the numerically computed condition numbers for the two-grid method for the first domain $\Omega=(0,1)^2$. The two-grid method has excellent condition numbers, except for the meshes corresponding to $\tfrac{\kappa h}{\tilde{p}}\approx 1$. As mentioned in Remark \ref{chap4:remark:Helmholtz-kernel}, the larger condition numbers for these meshes are probably due to the presence of a large Helmholtz kernel. 
\begin{figure}[h!]
    \centering
    \includegraphics[width=0.5\linewidth]{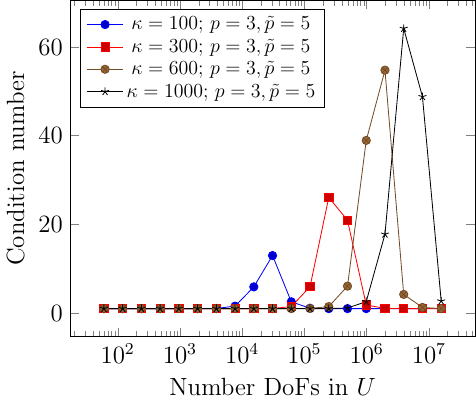}
\caption{Condition number using the two-grid preconditioner for the first example, for different values of $\kappa$.}
\label{chap4:fig:example1_twogrid}
\end{figure}

Motivated by the observations from Figure \ref{chap4:fig:example1_twogrid} we will now change the value of $m_\ell$, as defined in Section \ref{chap4:multigrid}, for meshes $\tria_\ell$ with mesh-size $h\approx \tfrac{\tilde{p}}{\kappa}$. In Figure \ref{chap4:fig:example1_different_m} we report on the results for the example from Section \ref{chap4:example1}, with $\kappa = 1000$, $p=3$ and $\tilde{p}=5$. This figure shows that the condition numbers decrease on meshes where $m_\ell$ is modified, which also leads to an overall reduction of condition numbers on finer meshes where we kept $m_\ell=1$ unchanged.

\begin{figure}[h!]
    \centering
    \includegraphics[width=0.5\linewidth]{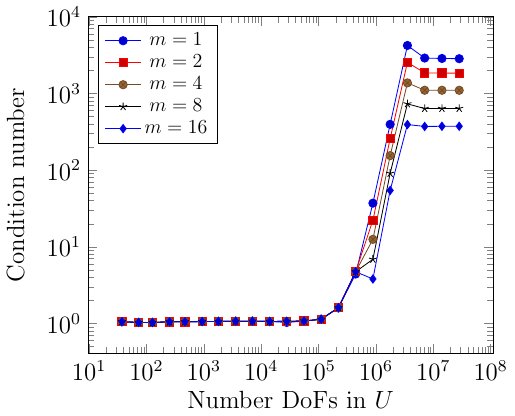}
\caption{Condition number of the preconditioned system for the example from Section \ref{chap4:example1}, for $\kappa = 1000$, $p=3$ and $\tilde{p}=5$, where we set $m_\ell = m$ for $\tria_{15}, \tria_{16}$ and $\tria_{17}$.}
\label{chap4:fig:example1_different_m}
\end{figure}

\FloatBarrier
\subsection{Solving the Helmholtz problems}\label{chap4:sec:solving}
In this section we report on the number of MINRES iterations needed to solve the Helmholtz equation. We will investigate the relation between the number of iterations needed and the wave number $\kappa$. 

Before we tackle the more difficult scattering domains, we first demonstrate the use of different stopping criteria and prolongation of solutions from previous meshes for the first example on the unit square. Combining both ideas, the number of iterations needed is significantly reduced.

\subsubsection{Stopping criteria}
For the example from Section \ref{chap4:example1} we will use two different stopping criteria. Firstly, we will stop when the initial residual has decreased by a factor of $10^8$, which we call stopping criterion $1$. Secondly, with stopping criterion $2$, the iteration is stopped when the approximate algebraic error is less than half the approximate total error, where we apply the approach outlined in Section~\ref{chap4:section:stoppingcriteria} to estimate both these errors. The starting vector in both cases is set to zero. Here, and in the rest of this section, the preconditioner for the Schur complement ${\bf Q}_S$ is chosen such that the spectrum of $({\bf Q}_S)^{-1} {\bf M}^U$ is contained in the interval $[0.9,1.1]$, which is achieved as described in Remark \ref{chap4:remark:Chebyshev}.

In Figure \ref{chap4:fig:comparestopping} the iteration numbers for both stopping criteria are shown. They seem to be growing linearly with $\kappa$, though fewer iterations are required for the second stopping criterion.

\begin{figure}[h!]
    \begin{subfigure}{.5\textwidth}
    \centering
    \includegraphics[width=\linewidth]{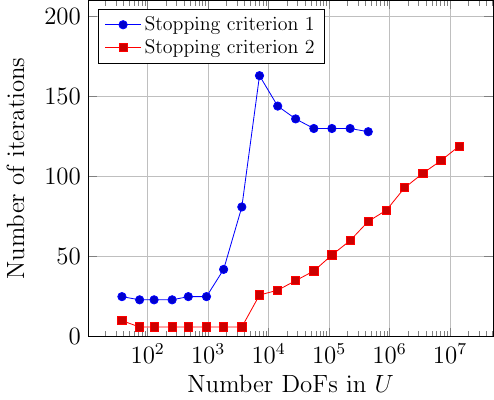}
    \end{subfigure}\hspace*{0cm}%
    \begin{subfigure}{.5\textwidth}
    \centering
    \includegraphics[width=\linewidth]{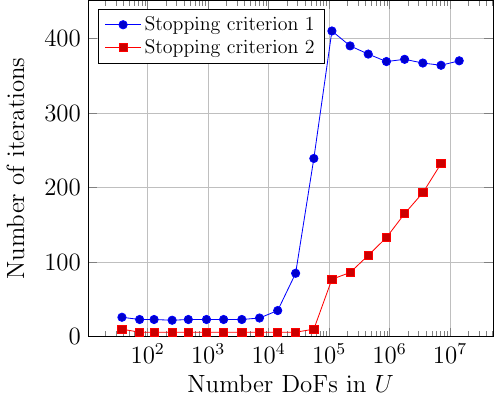}
    \end{subfigure}
    
     \begin{subfigure}{.5\textwidth}
    \centering
    \includegraphics[width=\linewidth]{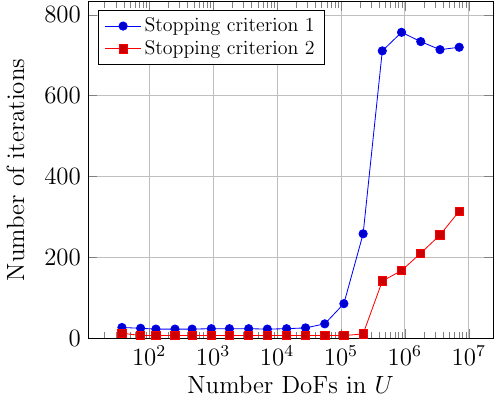}
    \end{subfigure}\hspace*{0cm}%
    \begin{subfigure}{.5\textwidth}
    \centering
    \includegraphics[width=\linewidth]{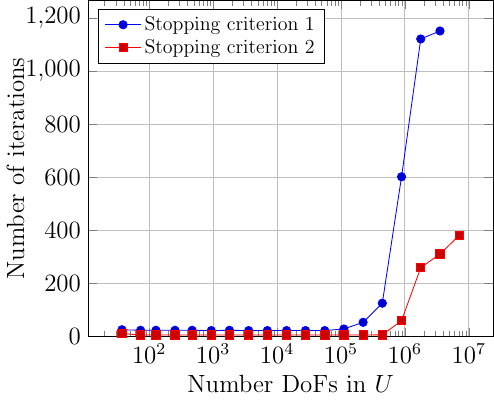}
    \end{subfigure}
\caption{Number of MINRES iterations for the example from Section \ref{chap4:example1}, for different values of $\kappa$. Left-upper: $\kappa=100$, right-upper: $\kappa = 300$, left-lower: $\kappa = 600$, right-lower: $\kappa = 1000$}
\label{chap4:fig:comparestopping}
\end{figure}

The second stopping criterion, however, does not terminate the MINRES solver prematurely. In Figure \ref{chap4:fig:error_against_iteration}, we show the error in the $U$-norm against the number of MINRES iterations for a generic case. In this figure, the vertical dashed line indicates when the second stopping criterion was met. Remarkably, Figure \ref{chap4:fig:error_against_iteration} shows that this stopping criterion was met at the exact moment the error in the $U$-norm reached its minimum. Furthermore, notice that the residual estimator $\|B_\kappa ' \tilde{\mathbbm{v}}^\delta\|_U$ converged much earlier than $\tilde{\mathbbm{u}}^\delta$, which was also mentioned in Section \ref{chap4:section:stoppingcriteria}. 

\begin{figure}[h!]
    \centering
    \includegraphics[width=0.5\linewidth]{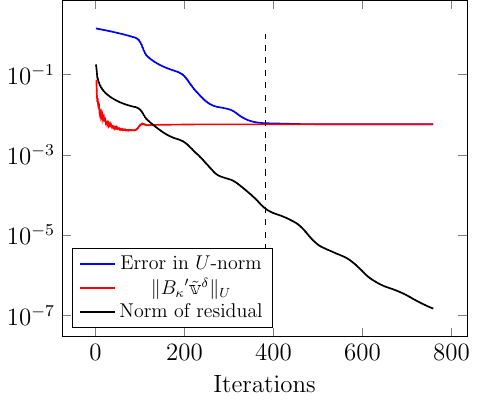}
\caption{Error in the $U$-norm, the size of the error estimator $\|{B_\kappa}' \tilde{\mathbbm{v}}^\delta\|_U$ and norm of the residual of the matrix-vector equation against the number of MINRES iterations for the example from Section \ref{chap4:example1}, with $\kappa=1000$, $p = 3$, $\tilde{p} = 5$ and number of DoFs in $U^\delta$ is $7087107$. The vertical dashed line indicates the iteration number where the second stopping criterion was met; beyond this point the error in the $U$-norm decreased only marginally.}
\label{chap4:fig:error_against_iteration}
\end{figure}

\subsubsection{Prolongating solutions}\label{chap4:sec:prolongate}
Rather than using the zero vector as the initial guess for the MINRES method, we can further reduce the number of iterations by prolongating the solution from the previous mesh. This approach, however, requires careful handling of the second stopping criterion. Since only a few iterations will be needed on finer meshes, the approximation of $\gamma_{V_\kappa}$ using harmonic Ritz values, as discussed in Section \ref{chap4:section:stoppingcriteria}, becomes less accurate. Consequently, the second stopping criterion may become unreliable.  

To address this issue, we adopt the following strategy: on each mesh, we carry over the approximation of $\gamma_{V_\kappa}$ from the previous mesh and update it only if a newly computed approximation using harmonic Ritz values is lower. This approach is justified by the observation that $\gamma_{V_\kappa}$ remains nearly constant for meshes with mesh-size $h < \tfrac{\tilde{p}}{\kappa}$, while for meshes where $h \approx \tfrac{\tilde{p}}{\kappa}$ we are still performing enough iterations to obtain a reasonable estimate for $\gamma_{V_\kappa}$.

In Figure \ref{chap4:fig:prolongate_example1} the iteration numbers are shown when we prolongate the solutions from previous meshes and use the stopping criterion from Section \ref{chap4:section:stoppingcriteria} together with the modifications described above. Again, the number of iterations appears to be growing linearly with $\kappa$. However, for fixed $\kappa$, the number of iterations now decreases with decreasing ${h<\frac{\tilde p}{\kappa}}$.

\begin{figure}[h!]
\centering
    \includegraphics[width=0.5\linewidth]{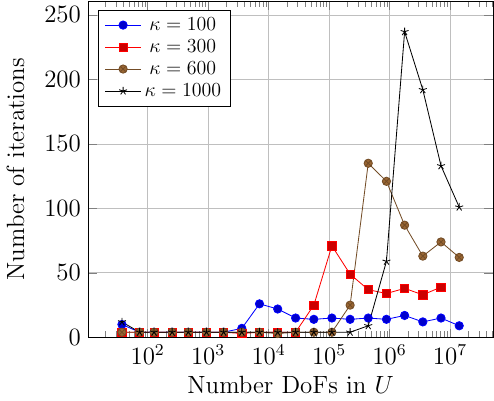}
\caption{Number of MINRES iterations against the number of DoFs in $U^\delta$ for the example from Section \ref{chap4:example1}, with stopping criterion $2$ and prolongated solutions, for $\kappa=100,300,600$ and $1000$, $p = 3$, $\tilde{p} = 5$.}
\label{chap4:fig:prolongate_example1}
\end{figure}

\subsubsection{Solving the Helmholtz equation on scattering domains}
We now will solve the Helmholtz equation on the scattering domains from Section \ref{chap4:example2} and Section \ref{chap4:example3}. We employ stopping criterion $2$ together with the prolongation of solutions as described in Section \ref{chap4:sec:prolongate}. 

In order to obtain optimal convergence in the $U$-norm, the solutions to these problems should be obtained using adaptive mesh refinement. Indeed, for the example from Section \ref{chap4:example2} with $\kappa=300$, in case of uniform refinement the error estimator was equal to $0.0106$ on a mesh with $\numprint{3988224}$ DoFs in $U^\delta$, while in case of adaptive refinement the error estimator was already equal to $0.0026$ on a mesh with $\numprint{4023945}$ DoFs in $U^\delta$.

Figure \ref{chap4:fig:iters} shows the number of iterations needed to solve the Helmholtz problem using adaptive refinement. To keep the inf-sup constants $\gamma_\kappa^\delta$ uniformly bounded, we increased the polynomial order of the test space for $\kappa=300$ from $\tilde{p}=5$ to $\tilde{p}=6$ for both scattering examples. We observe that the number of iterations grows about linearly as a function of the wave number $\kappa$.

To determine whether the current stopping criterion does not terminate the MINRES prematurely, we compare our solution with another solution produced by the MINRES method where we iterate until the approximate algebraic error is $20$ times smaller than the approximate total error (the approximate algebraic error and the approximate total error are defined in Section~\ref{chap4:section:stoppingcriteria}). In Figure \ref{chap4:fig:algebraic error} we plot the difference in the $U$-norm between both iterative solutions, divided by the error estimator $\|B_\kappa'\mathbbm{v}^\delta\|_U$. 
We observe that the difference in the $U$-norm is not always less than half the error estimator. This happens because our estimate of the constant $c$ introduced in Section \ref{chap4:sec:errorestimation} is not exact, primarily due to slightly overestimating $\gamma_{V_\kappa}$. Furthermore, $\gamma_\kappa^\delta$ is not exactly equal to $1$. Nevertheless, the performance of the stopping criterion appears to be robust with respect to the wave number $\kappa$. In addition, the difference in the $U$-norm is always smaller than $1$ on resolved meshes. 

For the trapping domain of Section \ref{chap4:example3}, we observe a drop in the number of iterations from $\kappa=100$ to $\kappa=200$. This is probably due to an overestimation of $\gamma_{V_\kappa}$ for $\kappa=50$ and $\kappa = 100$.

\begin{figure}[h!]
    \begin{subfigure}{.5\textwidth}
    \centering
    \includegraphics[width=\linewidth]{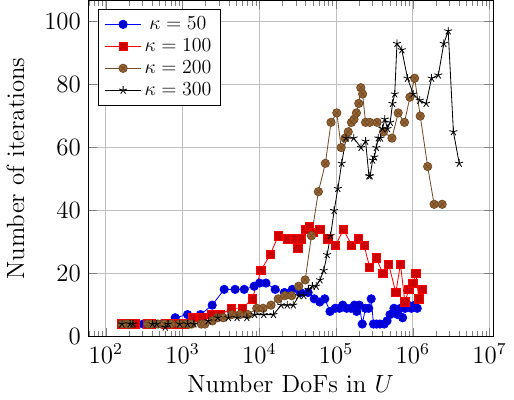}
    \end{subfigure}\hspace*{0cm}%
    \begin{subfigure}{.5\textwidth}
    \centering
    \includegraphics[width=\linewidth]{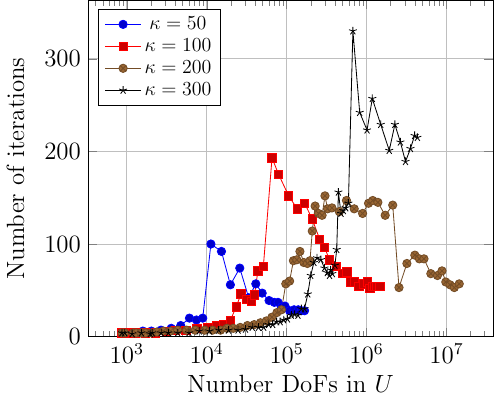}
    \end{subfigure}
\caption{Number of MINRES iterations against the number of DoFs in $U^\delta$ for different values of $\kappa$. Left: for the example from Section \ref{chap4:example2}, right: for the example from Section \ref{chap4:example3}.}
\label{chap4:fig:iters}
\end{figure}
\begin{figure}[h!]
    \begin{subfigure}{.5\textwidth}
    \centering
    \includegraphics[width=\linewidth]{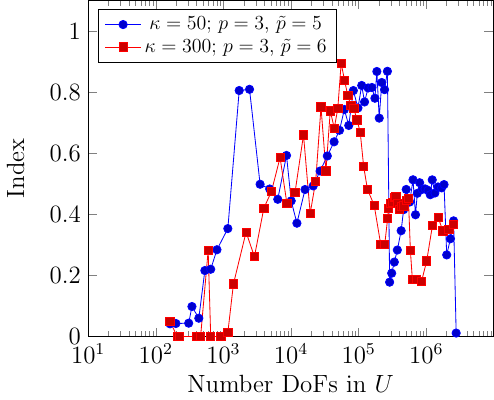}
    \end{subfigure}\hspace*{0cm}%
    \begin{subfigure}{.5\textwidth}
    \centering
    \includegraphics[width=\linewidth]{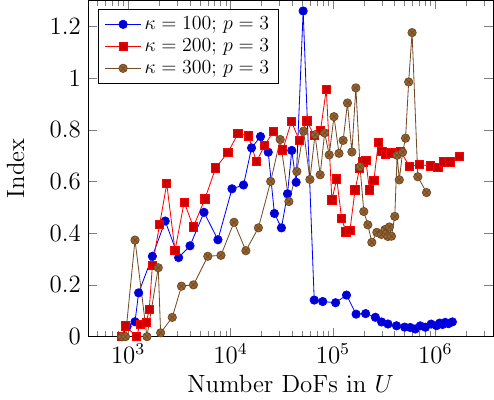}
    \end{subfigure}
\caption{The estimated algebraic error at termination of the MINRES iteration divided by the error estimator $\|B_\kappa'\mathbbm{v}^\delta\|_U$ against the number of DoFs in $U^\delta$ for different values of $\kappa$. Left: for the example from Section \ref{chap4:example2}, right: for the example from Section \ref{chap4:example3}.}
\label{chap4:fig:algebraic error}
\end{figure}

\FloatBarrier
\section{Conclusion}

In this article, we iteratively solved a pollution-free FOSLS formulation of the Helmholtz equation. Upon discretization of this FOSLS we obtained a saddle-point system for which we designed a block-preconditioner. This block preconditioner consists of two preconditioners: one for the Schur complement, which was straightforward to design, and one for the Riesz operator corresponding to the test space $V_\kappa^\delta$ equipped with the optimal test norm. To ensure that the latter preconditioner is Hermitian positive definite, we employed the theory of subspace corrections.

There are multiple advantages of the method described here. 
Firstly, this method is pollution-free, which is achieved by only slightly increasing the polynomial order on the test space if necessary.
Secondly, the use of general adaptive meshes is allowed. The preconditioner discussed in this article can be applied directly to any mesh without modifications and always leads to a convergent iterative solver. Being able to use adaptive meshes, we are able to recover optimal convergence rates for Helmholtz problems on scattering domains (this is demonstrated in \cite{204.18}). The error estimator, which is used to drive adaptive refinement, is obtained in terms of already computed quantities. 
Thirdly, it is possible to estimate the algebraic error accurately, which prevents unnecessary iterations.
Finally, the method is easy to implement and uses $\mathcal{O}(n)$ memory.  

Concerning the number of MINRES iterations needed, we observed a linear dependence on the wave number. Similar behavior has been reported in \cite{MR4630864} for $\Omega = (0,1)^2$, where a preconditioner is constructed for the DPG-method. However, our approach offers an advantage in terms of ease of implementation. For the (S)ORAS preconditioner, again for $\Omega=(0,1)^2$, a slightly better dependence has been reported for wave numbers up to $\kappa = 300$ in \cite{MR3644970}. However, there a coarse grid is used and the subdomain solves use more memory. Sweeping preconditioners for finite difference discretizations of the Helmholtz equation \cite{MR2818416} seem to produce good condition numbers of the preconditioned system for large $\kappa$, but are difficult to apply to unstructured grids or non-rectangular domains.

\subsection*{Acknowledgment}
The author wishes to thank his advisor Rob Stevenson for
the many helpful comments.

\end{document}